\documentclass[12pt,reqno]{amsart}

\usepackage[all]{xy}
\usepackage{amssymb}
\usepackage{parskip}
\usepackage{tikz,pgf}
\usepackage{mathtools}
\usepackage{booktabs}
\usepackage{fourier}
\usepackage[T1]{fontenc}
\usepackage[utf8]{inputenc} 
\usetikzlibrary{calc, matrix, positioning, backgrounds,arrows.meta}
\usepackage{hyperref}

\usepackage{xcolor}
\usepackage{color}
\def\red{}
\def\rojo{}
\def\blue{}

\def\card{\text{card}}
\def\espacio{~}

\newcommand{\nn}{\ensuremath{\mathbf{n}}}
\newcommand{\mdkn}{\ensuremath{M_d^{(k)}(n)}}
\newcommand{\adkn}{\ensuremath{A_d^{(k)}(n)}}

\newcommand\tinyelemental[6]
{%
    \node[rectangle] (#1) at (#2,#3) {
    \begin{tikzpicture}[baseline=(current bounding box.south)]
    \node (A) [rectangle, draw, minimum width=1cm, inner sep=0pt] at (0,0) {\ensuremath{#4}, \ensuremath{#5}\rule[-1.5mm]{0pt}{5.5mm}};
    \node (r) [circle, draw, minimum width=6mm,inner sep=1pt] at ([xshift=0.15cm,yshift=0.8cm]0,0) {\ensuremath{#6}};
    \draw[line width=0.5pt,-] (A.north-|r)--(r.south);
    \end{tikzpicture}};
}

\newcommand{\elementary}[3]{\tikz[baseline=-4pt,scale=1, every node/.style={transform shape}]
    \tinyelemental{A}{0}{0}{#1}{#2}{#3};}

\newtheorem{theorem}{Theorem}[section]
\newtheorem{proposition}[theorem]{Proposition}
\newtheorem{lemma}[theorem]{Lemma}
\newtheorem{corollary}[theorem]{Corollary}

\theoremstyle{definition}
\newtheorem{definition}[theorem]{Definition}
\newtheorem{example}[theorem]{Example}
 
\theoremstyle{remark}
\newtheorem{remark}[theorem]{Remark}

\tikzset{
  elementary/.code n args={3}{%
    \node (r) [circle, draw, minimum width=7mm,inner sep=1pt] at (0.15,1) {\ensuremath{#3}};
    \node (A) [rectangle, draw, minimum width=1cm] at (0,0) {\ensuremath{#1}, \ensuremath{#2}};
    \draw[line width=0.5pt,-] (A.north-|r)--(r.south);
    \node at (1,0) {$\cdot$};
  }}

\title[\red{The rational (non-)formality of the non-$\red{3}$-equal manifolds}]{\red{The rational (non-)formality of the non-$3$-equal manifolds}}
\thanks{\red{The first author is grateful to Jos\'e Cantarero and CIMAT M\'erida for the warm hospitality while pieces of this work went underway.} \blue{The second author thanks the CIMAT Mérida Algebraic Topology Group for providing an enriching environment to complete parts of this paper under the support of a CONAHCYT postdoctoral grant.
}}
\author{Jes\'us Gonz\'alez \ \ and \ \ Jos\'e Luis Le\'on-Medina}

\begin{document}

\begin{abstract}
\red{Let $M^{(\red{k})}_{d}(n)$ be the manifold of $n$-tuples $(x_1,\ldots,x_n)\in(\mathbb{R}^d)^n$ having non-$k$-equal coordinates. We show that, for $d\geq2$, $M^{(\red{3})}_{d}(n)$ is rationally formal if and only if $n\leq6$. This stands in sharp contrast with the fact that all classical configuration spaces $M^{(\red{2})}_d(n)=\text{Conf\hspace{.4mm}}(\hspace{.2mm}\mathbb{R}^d,n)$ are rationally formal, just as are all complements of arrangements of arbitrary complex subspaces with geometric lattice of intersections. The rational non formality of $M^{(\red{3})}_{d}(n)$ for $n>6$ is established via detection of non-trivial triple Massey products assessed through Poincar\'e duality.} 
\end{abstract}

\maketitle



\section{Introduction and main results}\label{seccionintroduccion}
The non-$k$-equal manifold $\mdkn$ \red{is defined as the} complement in $(\mathbb{R}^d)^n$ of the arrangement $\adkn$ \red{consisting} of the diagonal subspaces
\[
A_I = \big\lbrace(x_1,\dots,x_n)\in \big(\mathbb{R}^d\big)^n \,\big\lvert\, x_{i_1} = \cdots = x_{i_k}\big\rbrace,
\]
where $I=\{i_1,\dots, i_k\}$ runs through the cardinality-$k$ subsets of the segment $\nn=\{1,2,\dots, n\}$. \red{Note that $\mdkn=(\mathbb{R}^d)^n$ for $k>n$, while $M^{(n)}_d(n)$ has the homotopy type of the $(dn-d-1)$-dimensional sphere. \blue{Hence, i}n this paper, we focus on the more interesting cases with $k<n$. We will also assume implicitly that $d\geq2$, so to stay \blue{a}way from the non-simply connected ---but aspherical--- spaces $M^{(3)}_1(n)$, and that $k\geq3$, so to stay away from the classical configuration spaces Conf$(\mathbb{R}^d,n)$.}

In \cite{Miller2012}, Miller obtained partial results on the structure of Massey products \red{in} the ``complex'' case, i.e., \red{for} non-$k$-equal manifolds \red{$\mdkn$} with $d=2$. \red{Miller} showed that all $p$-order Massey products on $M^{(k)}_2(n)$ vanish provided either one of the following conditions holds:
\begin{align}
3&\leq p<k\qquad \mbox{\hspace{-.5mm} (in view of \cite[Theorem 1.2]{Miller2012}).} \label{loworder} \\
n&\leq k(k-1)\quad \mbox{(in view of \cite[Corollary 1.3]{Miller2012}).} \label{lowparticles}
\end{align}
Both of these restrictions are \red{sharp as} neither the upper bound for $p$ in~(\ref{loworder}) nor the upper bound for $n$ in~(\ref{lowparticles}) can be improved for general $k$. Indeed, Miller proved in addition that
\begin{equation}\label{firsttriple}
\mbox{$M^{(3)}_2(n)$ admits non-trivial triple Massey products when $n>6$.}
\end{equation}

\red{This paper is motivated by the fact that the statement in (\ref{firsttriple}) turns out to be sharp also in rational-formality terms. In short: $M^{(3)}_2(n)$ cannot admit non-trivial triple Massey products for $n\leq6$ as, in fact, such a space is rationally formal, as recorded in (\ref{ratfor}) below. Indeed, start by recalling that any $c$-connected CW complex $X$ with $c\geq1$ and
\begin{equation}\label{conectividad}
\dim(X)\leq 3c+1
\end{equation}
is rationally formal (see \cite[Corollary 5.16]{Halperin1979}). Then, as observed in the proof of \cite[Theorem 3.3]{MR4421862}, for $d\geq2$ and $k\geq3$, the connectivity condition above holds for $\mdkn$ with $c=d(k-1)-2$, while (\ref{conectividad}) can be spelled out as}
\begin{equation}\label{dimcond}
\red{n(d-1)+m(k-2)\leq 3dk-2(d+3).}
\end{equation}
\red{Here and below $m$ stands for the integral part of $n/k$. Note that when $k=3$, condition (\ref{dimcond}) holds if and only if $n\leq6$ (recall $d\geq2$). The point then is that
\begin{equation}\label{ratfor}
\mbox{$M^{(3)}_d(n)$ is rationally formal for $n\leq6$.}
\end{equation}
However, $M^{(3)}_d(n)$ is not rationally formal for $n>6$, at least when $d=2$, in view (\ref{firsttriple}). The goal \blue{of} this paper is to show that the restriction $d=2$ in the latter assertion is\blue{,} in fact\blue{,} unnecessary:}
\begin{theorem}\label{driving0}
\red{For $d\geq2$, the non-3-equal manifold $M^{(3)}_d(n)$ is rationally formal if and only if $n\leq6$.}
\end{theorem}

\red{Theorem \ref{driving0} stands in contrast with the known fact\footnote{\red{The formality of $\text{Conf}(\mathbb{R}^d,n)$ is proved by Arnold in \cite{Arn69} for $d=2$ through an \emph{ad-hoc} argument using complex analysis, and by Kontsevich in \cite{MR1718044} for general $d\hspace{.2mm}$, as part of a simplification of Tamarkin's algebraic proof of Kontsevich Formality Theorem. The formality of complements of the indicated complex arrangements is proved by Feichtner and Yuzvinsky in \cite{MR2183223} through direct construction of explicit quasi-isomorphisms among various rational models, the last of which is shown to be formal.}} that all classical configuration spaces $M^{(2)}_d(n)=\text{Conf}(\mathbb{R}^d,n)$ are rationally formal, just as are all complements of arrangements of arbitrary complex linear subspaces with geometric lattice of intersections.}

\red{In view of (\ref{ratfor}), Theorem \ref{driving0} will be proved once we show the rational non-formality of $M^{(3)}_d(n)$ for $n>6$. Such a task will be achieved through the identification of non-trivial triple Massey products in $M^{(3)}_d(n)$ (Theorem~\ref{ptmnt} and Corollary \ref{muygeneral} at the end of the paper).}

Miller's paper ends by conjecturing that
\begin{equation}\label{millersconjecture}
\mbox{\emph{$M^{(k)}_{\red{d}}(n)$ has non-trivial Massey products for $n>k(k-1)$ \red{and $d=2$.}}}
\end{equation}
Presumably, the non-trivial Massey products conjectured by \red{Miller would} be of order precisely $k$. \red{In view of the results in this paper, it is tempting to think that the Poincar\'e-duality technique could actually lead to a proof of (\ref{millersconjecture}), even without the restriction on $d$.}

Miller's computations \red{of} cup-\red{products} and Massey-product\red{s on the complex manifolds $M^{(k)}_2(n)$} are based on Yuzvinsky's DGA structure on the relative atomic complex for $M^{(k)}_d(n)$ introduced by Vassiliev (\cite{Vassiliev1993,Yuzvinsky99}). With such an approach, much effort is required to show \red{cohomological} non-triviality of \red{a given cocycle. As} a consequence, the extent of results in~\cite{Miller2012} get somehow limited. \red{We have circumvented} the problem by following a more direct route. Namely, as originally noted in~\cite{Massey1968}, in many cases Poincar\'e duality and intersection theory (using Borel-Moore homology in our non-compact case) can be used to evaluate Massey products. Actually, Dobrinskaya and Turchin use Poincar\'e duality in~\cite{dobri2015} to give a fully workable description of the cohomology ring of $\mdkn$. \red{Here\blue{,} we build on their approach} in order to \red{get an effective assessment of} Massey products. 

\red{We have not touched the classification by rational formality of manifolds $\mdkn$ with $k>3$, as such a problem is far more involved, possibly not within reach with current technology. Indeed, assuming $k>3$, we see that} condition ~(\ref{lowparticles}) is \red{strictly} less restrictive than the case $d=2$ of\espacio(\ref{dimcond}). \red{So,} \rojo{even for $d=2$,} there are manifolds $M^{(k)}_{\blue{d}}(n)$ which, \red{despite not supporting non-trivial Massey products of any order, are not known to be rationally formal through a simple application of\espacio(\ref{dimcond}) (or \blue{through} any other method known to the authors). Even worst, the gap would not improve much even if (\ref{conectividad}) ---and therefore (\ref{dimcond})--- could be improved by a condition of the sort $\dim(X)\leq 4c+\varepsilon$, i.e., a potential analogue (for non-compact manifolds!) of the results in \cite{MR0528561, MR2238645}.}

\red{As noted above, the hypotheses $d\geq2$ and $n>k\geq3$ will be in force throughout the paper.} 

\section{\red{The cohomology of \texorpdfstring{$\mdkn$}{non-k-equal manifolds}: Additive structure}}\label{secciondescripciondecohomologia}
In this section we recall the \red{geometric-combinatorial} description of the \red{additive structure of the} cohomology of $\mdkn$, \red{as} given in~\cite{dobri2015}. \rojo{We also shed additional light on some points in Dobrinskaya-Turchin's constructions.} All \red{homology and} cohomology groups will be taken with \red{either integer ($\mathbb{Z}$) or mod-2 ($\mathbb{Z}_2$) coefficients.} Assertions made without specifying \red{coefficients} are meant to hold for both options. \red{While $\mathbb{Z}$ coefficients are needed to set descriptions correctly, Massey product computations in Section \ref{mps} will use exclusively mod 2 coefficients for the sake of simplicity, so orientations and sign specifications below can and will safely} be ignored.

\begin{definition}\label{defkforesto}
\begin{itemize}
\item[\red{(a)}] A $k$-forest on $\nn$ (or simply a $k$-forest) is an acyclic simple graph \red{which is \emph{$\nn$-bipartitioned\hspace{.3mm}} in the sense that it has} two types of vertices, square ones and round ones, each containing a certain subset of\espacio$\nn\hspace{.2mm}$, \red{and in such a way that the subsets of integers inside the various vertices partition $\nn$}. A square vertex must contain $k-1$ elements of $\nn$, and cannot be an isolated vertex. \red{In} fact, the set of immediate neighbors of a square vertex must contain a round vertex. A round vertex must contain a single element of $\nn$, and must be either an isolated vertex or have valency 1, in which case it must be connected to a square vertex. Square vertices are declared to have degree $d(k-2)$, while edges are declared to have degree $d-1$. The degree \red{$\deg(T)$} of a $k$-forest \red{$T$} is then defined as the sum of the degrees of \red{the} square vertices and edges \red{of $T$}.
\item[\red{(b)}] An orientation for a $k$-forest consists of three ingredients:
\begin{itemize}
    \item[\red{(b.1)}] An orientation for each edge;
    \item[\red{(b.2)}] A total ordering for the elements inside each square vertex;
    \item[\red{(b.3)}] A total ordering for the \emph{orientation set}, i.\,e., the set consisting of all edges and all square vertices.
\end{itemize}
\end{itemize}
\end{definition} 

\begin{theorem}[{\cite[Theorem 6.1]{dobri2015}}]\label{sum} Let \red{$R\in\{\mathbb{Z},\mathbb{Z}_2\}$. As a graded $R$-module,} $H^*(\mdkn)$ is free and generated by oriented $k$-forests on\espacio$\hspace{.1mm}\nn$ subject to the relations listed below.
\begin{enumerate}
\item Orientation relations:
    \begin{itemize}
        \item[(i)] Permuting the order of the orientation set introduces \red{a} Koszul sign induced by the permutation (with respect to the degrees of the elements of the orientation set).
        \item[(ii)] A permutation $\sigma \in \Sigma_{k-1}$ of the elements inside a square vertex introduces the sign $\epsilon(\sigma)^{d}$, where $\epsilon(\sigma)$ stands for the sign of $\sigma$.
        \item[(iii)] Reversing the orientation of an edge introduces the sign $(-1)^d$.
    \end{itemize}
\item Three-term relations:
    \begin{center}
    \scalebox{.92}{\begin{tikzpicture}
    \node (A) [rectangle, draw, minimum width=8mm] at (-1,0) {$A$};
    \node (B) [rectangle, draw, minimum width=8mm] at (0,1) {$B$};
    \node (C) [rectangle, draw, minimum width=8mm] at (1,0) {$C$};
    \draw[line width=0.5pt,-latex] (A.north)--(B.south west)node[midway, xshift=-2mm, yshift=1mm]{\tiny $1$}; 
    \draw[line width=0.5pt,-latex] (B.south east)--(C.north)node[midway,xshift=2mm,yshift=1mm]{\tiny $2$};
    \begin{scope}[xshift=3.75cm]
    \node (A2) [rectangle, draw, minimum width=8mm] at (-1,0) {$A$};
    \node (B2) [rectangle, draw, minimum width=8mm] at (0,1) {$B$};
    \node (C2) [rectangle, draw, minimum width=8mm] at (1,0) {$C$};
    \draw[line width=0.5pt,latex-] (A2.east)--(C2.west)node[midway,below]{\tiny $2$}; 
    \draw[line width=0.5pt,-latex] (B2.south east)--(C2.north)node[midway, xshift=1mm, yshift=1.75mm] {\tiny $1$};
    \end{scope}
    \begin{scope}[xshift=7.5cm]
    \node (A3) [rectangle, draw, minimum width=8mm] at (-1,0) {$A$};
    \node (B3) [rectangle, draw, minimum width=8mm] at (0,1) {$B$};
    \node (C3) [rectangle, draw, minimum width=8mm] at (1,0) {$C$};
    \draw[line width=0.5pt,-latex] (A3.north)--(B3.south west)node[midway,xshift=-1.25mm, yshift=1.8mm]{\tiny $2$};
    \draw[line width=0.5pt,latex-] (A3.east)--(C3.west)node[midway,below]{\tiny $1$};
    \end{scope}
    \coordinate (D) at ($(C.east)!0.5!(A2.west)$);
    \coordinate (E) at ($(C2.east)!0.5!(A3.west)$);
    \node at ([yshift=6mm]D) {$+$};
    \node at ([yshift=6mm]E) {$+$};
    \node at ([shift={(-110mm,6mm)}]C3.east) {$0\hspace{3mm}=$};
    \end{tikzpicture}}
    \end{center}
    The three pictures are local in the sense that we have three oriented $k$-forests that are identical except for the disposition \red{and ordering of the two} oriented edges connecting vertices $A$, $B$ and $C$. \red{The relative orderings of each such pair of oriented edges} in the corresponding orientation sets are indicated by the attached numbers.
\item Generalized Jacobi relations:
    \begin{center}
    \begin{tikzpicture}
    \node at (-.42,-1.5) {$\cdots$}; \node at (1.77,-1.5) {$\cdots$};
    \node (rec) [rectangle, draw] at (0,0) {$i_1$ $i_2$ $\cdots$ $i_{k-2}$ $j_\ell$};
    \node (r1) [circle, draw, inner sep=0.5pt,minimum width=7.5mm] at (-2,-1.5) {$j_{\scriptscriptstyle 1}$};
    \node (r2) [circle, draw, inner sep=0.5pt,minimum width=7.5mm] at (-1.1,-1.5) {$j_{\scriptscriptstyle 2}$};
    \node (r3) [circle, draw, inner sep=0.5pt,minimum width=7.5mm] at (.2,-1.5) {$j_{\scriptscriptstyle\ell-1}$};
    \node (r4) [circle, draw, inner sep=0.5pt,minimum width=7.5mm] at (1.1,-1.5) {$j_{\scriptscriptstyle\ell+1}$};
    \node (r5) [circle, draw, inner sep=0.5pt,minimum width=7.5mm] at (2.4,-1.5) {$j_{\scriptscriptstyle\red{\omega}}$};
    \draw[line width=0.5pt,-latex] (rec.south west)--(r1.north)node[midway,xshift=.8mm,yshift=3.3mm]{\tiny $1$};
    \draw[line width=0.5pt,-latex] ($(rec.south west)!0.5!(rec.south)$)--(r2.north)node[midway,xshift=-.1mm,yshift=2.5mm]{\tiny $2$};
    \draw[line width=0.5pt,-latex] (rec.south)--(r3.north)node[midway,xshift=-5mm,yshift=-0.5mm]{\tiny $\cdots$};
    \draw[line width=0.5pt,-latex] ($(rec.south east)!0.5!(rec.south)$)--(r4.north);
    \draw[line width=0.5pt,-latex] (rec.south east)--(r5.north)node[midway,xshift=-5mm,yshift=-0.5mm]{\tiny $\cdots$}node[midway,xshift=.45mm,yshift=3.3mm]{\tiny $\red{\omega}-1$};
    \node at (-3.5,-0.75) {$0\hspace{2mm}=\hspace{2mm}\displaystyle \sum_{\ell=1}^{\red{\omega}} (-1)^{\ell(d-1)}$};
    \end{tikzpicture}
    \end{center}
    \red{The $\omega$} pictures are again local. Moreover, in each of the global pictures, the square vertex cannot be connected to other (non\blue{-}shown) round vertices.
\end{enumerate}
\end{theorem}

In pictures like the one above, we agree that elements inside a square vertex are written increasingly, from left to right, following \red{ingredient (b.2) of the intended orientation.} \red{Note that,} in the orientation set, the transposition of a square vertex and an \red{oriented} edge produces a positive Koszul sign since $d(k-2)(d-1)$ is even. Thus, the ordering in the orientation set is really a pair of orderings, one for square vertices and another for \red{oriented} edges.

Orientation and three-term relations can be used to express any oriented $k$-forest as a linear combination of oriented \emph{\red{linear}} $k$-forests, i.e., oriented $k$-forests whose non-trivial\footnote{\red{In an oriented $k$-forest, a (connected) component that reduces to an isolated round vertex is said to be trivial.}} components are trees with square vertices lying along an embedded arc, as in Figure~\ref{semilinear}.
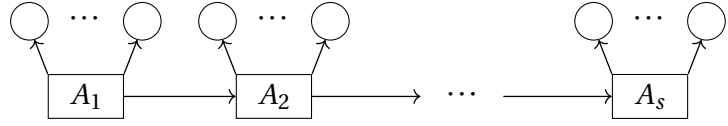
\begin{figure}[ht]
\centering
\begin{tikzpicture} 
\node (r1) [circle, draw, minimum width=5mm] at (-0.75,1) {};
\node (r2) [circle, draw, minimum width=5mm] at (0.75,1) {};
\node (A0) at (5,0) {$\;\;\;\cdots\;\;\;$};
\node (A1) [rectangle, draw, minimum width=1cm] at (0,0) {$A_1$};
\draw[->, line width=0.5pt] (A1.north west)--(r1.285);
\node at (0,1) {$\cdots$};
\draw[->, line width=0.5pt] (A1.north east)--(r2.255);
\begin{scope}[xshift=2.5cm]
\node (r3) [circle, draw, minimum width=5mm] at (-0.75,1) {};
\node (r4) [circle, draw, minimum width=5mm] at (0.75,1) {};
\node (A2) [rectangle, draw, minimum width=1cm] at (0,0) {$A_2$};
\draw[->, line width=0.5pt] (A2.north west)--(r3.285);
\node at (0,1) {$\cdots$};
\draw[->, line width=0.5pt] (A2.north east)--(r4.255);
\end{scope}
\draw[->, line width=0.5pt] (A1.east)--(A2.west);
\begin{scope}[xshift=7.5cm]
\node (r5) [circle, draw, minimum width=5mm] at (-0.75,1) {};
\node (r6) [circle, draw, minimum width=5mm] at (0.75,1) {};
\node (A3) [rectangle, draw, minimum width=1cm] at (0,0) {$A_s$};
\draw[->, line width=0.5pt] (A3.north west)--(r5.285);
\node at (0,1) {$\cdots$};
\draw[->, line width=0.5pt] (A3.north east)--(r6.255);
\end{scope}
\draw[->, line width=0.5pt] (A2.east)--(A0.west);
\draw[->, line width=0.5pt] (A0.east)--(A3.west);
\end{tikzpicture}
\caption{A non-trivial component of a linear $k$-forest.}
\label{semilinear}
\end{figure}
\red{Similarly,} orientation and generalized Jacobi relations can be used to \red{express any oriented} $k$-forest as a linear combination of \red{oriented} \emph{\red{ordered\hspace{.2mm}}} $k$-forests, i.e., those satisfying that the largest of the integers inside round vertices attached to \red{a given} square vertex \red{$A$} is larger than any of the integers inside \red{$A$.} \red{The two rewriting processes can then be coordinated so to yield basis elements:}

\begin{definition}\label{bosquesbasicos}
An \red{oriented} ordered linear $k$-forest is called \emph{basic} provided its non-trivial components satisfy the following conditions, where \red{notation is as in} Figure~\ref{semilinear}:
\begin{itemize}
\item Edge orientations are as indicated in Figure~\ref{semilinear}.
\item \blue{According to their orientation order: }
 $A_1<A_2<\cdots <A_s$.
\item For a portion of the form
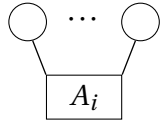
\begin{center}
\begin{tikzpicture} 
\node (r1) [circle, draw, minimum width=5mm] at (-0.75,1) {};
\node (r2) [circle, draw, minimum width=5mm] at (0.75,1) {};
\node (A1) [rectangle, draw, minimum width=1cm] at (0,0) {$A_i$};
\draw[line width=0.5pt] (A1.north west)--(r1.285);
\node at (0,1) {$\cdots$};
\draw[line width=0.5pt] (A1.north east)--(r2.255);
\end{tikzpicture}
\end{center}
the elements inside the square vertex appear in their natural order. Likewise, the ordering (in the orientation set) of the edges attaching round vertices to the square vertex agrees with the natural order of the integers inside the round vertices. Furthermore, if $i>1$, then the edge from $A_{i-1}$ to $A_i$ is smaller than all edges connecting $A_i$ to round vertices. Likewise, if $i<s$, then the edge from $A_i$ to $A_{i+1}$ is larger than all edges connecting $A_i$ to round vertices.
\item The minimal $m\in\nn$ inside the vertices of the linear tree component $C$ in Figure~\ref{semilinear} appears either inside $A_1$ or inside a round vertex attached to $A_1$. Furthermore, if $m'$ is the corresponding minimal element in another linear tree component $C'$ of the $k$-forest, and $m<m'$, then orientation elements associated to $C$ are smaller than orientation elements associated to~$C'$.
\end{itemize}
\end{definition}

\addtocounter{theorem}{-2}
\begin{theorem}[Continued]
Basic $k$-forests yield a graded basis for the cohomology of $\mdkn$.
\end{theorem}
\addtocounter{theorem}{1}

Theorem~\ref{sum} is proved in~\cite{dobri2015} in two stages. First, a set of cohomology classes parametrized by \red{oriented} $k$-forests is constructed as \red{Borel-Moore} Poincar\'e duals of fundamental classes of suitably chosen oriented submanifolds \red{of $\mdkn$. See the revision below. It is then checked that the cohomology classes resulting from basic $k$-forests} give the identity matrix when paired with an explicit basis for the homology of $\mdkn$. For the purposes of this work, \red{the rest of the section is devoted to recalling and illustrating the connection between oriented $k$-forests and their corresponding Poincar\'e-dual fundamental classes} in the Borel-Moore homology of $\mdkn$.

\red{Consider the projection $p_1:\mathbb{R}^d \to \mathbb{R}^{d-1}$ onto} the last $d-1$ coordinates, i.e., $p_1(x)=(x^{(2)},\ldots,x^{(d)})$, for $x=(x^{(1)},x^{(2)},\ldots,x^{(d)})$. \red{An oriented} $k$-forest $T$ determines \red{a convex domain $c_T$ of a vector subspace $C_T$ of $(\mathbb{R}^d)^n$. Explicitly, $C_T$ consists} of all tuples $(x_1,\ldots,x_n)\in(\mathbb{R}^d)^n$ satisfying
\begin{itemize}
    \item[(i)] \red{if} $i$ and $j$ in $\mathbf{n}$ lie in the same square vertex, then $x_i=x_j$\red{;}
    \item[(ii)] \red{if} two vertices $A$ and $B$ of $T$ are connected by an edge oriented from $A$ to $B$, then for all $i\in A$, $j\in B$, one has $p_1(x_i) = p_1(x_j)$.
\end{itemize}
Note that, if $i$ and $j$ lie in the same connected component of $T$, then the condition $p_1(x_i)=p_1(x_j)$ holds true for the points $(x_1,\ldots,x_n)$ in $C_T$. \red{The domain $c_T$, also referred to as a linear cell, is defined by the equalities above together with the inequalities
\begin{equation}\label{lado}
x_i^{(1)} \leq x_j^{(1)}
\end{equation}
in the case of (ii). Note that the degree of $T$, $\deg(T)$, in Definition~\ref{defkforesto} is the codimension of both $C_T$ and $c_T$ in $(\mathbb{R}^d)^n$.}

\red{The locally compact linear cell $c_T$ has boundary contained in $\adkn$ and, thus, represents} the Borel-Moore fundamental class \red{in $${H}_{dn-\deg(T)}^{\text{BM}}\big(\mdkn\big)$$ of the submanifold $\text{Int}(c_T)$ of $\mdkn$ given by the interior of $c_T$. We say that the submanifold $\text{Int}(c_T)$ is \emph{encoded} by $T$.}

\red{The orientation} ingredients of $T$ determine \red{(as illustrated below)} a co-orientation of $C_T$ \red{in $(\mathbb{R}^d)^n$ and, thus, of $\text{Int}(c_T)$ in $\mdkn$.} We thus get a \red{$\deg(T)$-dimensional} cohomology class in $\mdkn$ \red{which is Poincar\'e-dual to the Borel-Moore fundamental class of $\text{Int}(c_T)$ in $\mdkn$. As suggested by Theorem\espacio \ref{sum}, the resulting cohomology class is denoted\footnote{\red{While oriented $k$-forests can be considered as elements in the cohomology of $\mdkn$, two distinct forests might represent the same class, see for instance Remark\espacio\ref{cuadrote}. Of course, such a faithfulness problem does not hold in the case of basic $k$-forests.}} by $T$.}

For example, the \red{basic} $4$-forest $T\in H^8(M_{3}^{(4)}(7))$ given by
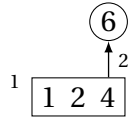
\begin{center}
\begin{tikzpicture}[baseline=(current bounding box.center)]
\node (A) [rectangle, draw, minimum width=10mm] at (-1,0) {$1\enskip 2\enskip 4$};
\node at (-1.85,.2) {\tiny$1$};
\node at (-.4,.48) {\tiny$2$};
\node (r1) [circle, draw, inner sep=0.5pt,minimum width=5mm] at (-0.6,1) {$6$};
\draw[line width=0.5pt,-latex] ([xshift=4mm]A.north)--(r1.south);
\end{tikzpicture}
\end{center}
corresponds to the \red{linear cell} $c_T$ consisting of all tuples $(x_1,\dots,x_7)\in(\mathbb{R}^3)^7$ such that $x_1 = x_2 = x_4$, $p_1(x_1)=p_1(x_6)$ and $x_1^{(1)} \leq x_6^{(1)}$. The co-orientation of \red{$C_T$,} i.e., the orientation of the normal bundle of $\red{C_T}\hookrightarrow(\mathbb{R}^3)^7$, is \blue{induced} through the surjection $\pi_T\colon(\mathbb{R}^3)^7\to\mathbb{R}^{\deg(T)}=(\mathbb{R}^3)^2\times\mathbb{R}^2$ with components $\pi_{\mbox{\tiny$\square$}}\colon(\mathbb{R}^3)^7\to(\mathbb{R}^3)^2$ and $\pi_{\mbox{\large$\circ$}}\colon(\mathbb{R}^3)^7\to\mathbb{R}^2$ given by $$\pi_{\mbox{\tiny$\square$}}(x_1,\ldots\blue{,} x_7)= (x_2-x_1,x_4-x_1)\mbox{ \;and \;} \pi_{\mbox{\large$\circ$}}(x_1,\ldots, x_7)=p_1(x_6-x_1).$$ Note that \blue{$C_T$} is the kernel of \red{$\pi_T$} \blue{and\rojo{,} hence\rojo{,}} the tangent space of $c_T$\blue{. Moreover, }following the orientation of $T$, the first and second components of $\pi_T$ account, respectively, for the square vertex and the edge \blue{in}\espacio$T$.

In such a setting, sums \red{of (co)homology classes} correspond to unions of \red{representing linear cells,} while signs arise from a consistent management \red{of orientations.} For example, \red{consider} the three-term relation
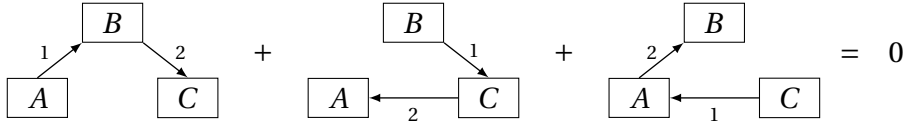
\begin{center}
    \begin{tikzpicture}
    \node (A) [rectangle, draw, minimum width=8mm] at (-1,0) {$A$};
    \node (B) [rectangle, draw, minimum width=8mm] at (0,1) {$B$};
    \node (C) [rectangle, draw, minimum width=8mm] at (1,0) {$C$};
    \draw[line width=0.5pt,-latex] (A.north)--(B.south west)node[midway, xshift=-2mm, yshift=1mm]{\tiny $1$}; 
    \draw[line width=0.5pt,-latex] (B.south east)--(C.north)node[midway,xshift=2mm,yshift=1mm]{\tiny $2$};
    \begin{scope}[xshift=4cm]
    \node (A2) [rectangle, draw, minimum width=8mm] at (-1,0) {$A$};
    \node (B2) [rectangle, draw, minimum width=8mm] at (0,1) {$B$};
    \node (C2) [rectangle, draw, minimum width=8mm] at (1,0) {$C$};
    \draw[line width=0.5pt,latex-] (A2.east)--(C2.west)node[midway,below]{\tiny $2$}; 
    \draw[line width=0.5pt,-latex] (B2.south east)--(C2.north)node[midway, xshift=1mm, yshift=1.25mm] {\tiny $1$};
    \end{scope}
    \begin{scope}[xshift=8cm]
    \node (A3) [rectangle, draw, minimum width=8mm] at (-1,0) {$A$};
    \node (B3) [rectangle, draw, minimum width=8mm] at (0,1) {$B$};
    \node (C3) [rectangle, draw, minimum width=8mm] at (1,0) {$C$};
    \draw[line width=0.5pt,-latex] (A3.north)--(B3.south west)node[midway,xshift=-1.25mm, yshift=1mm]{\tiny $2$};
    \draw[line width=0.5pt,latex-] (A3.east)--(C3.west)node[midway,below]{\tiny $1$};
    \end{scope}
    \coordinate (D) at ($(C.east)!0.5!(A2.west)$);
    \coordinate (E) at ($(C2.east)!0.5!(A3.west)$);
    \node at ([yshift=6mm]D) {$+$};
    \node at ([yshift=6mm]E) {$+$};
    \node at ([shift={(6.85mm,6mm)}]C3.east) {$=\hspace{3mm}0$};
    \end{tikzpicture}
    \end{center}
\red{which, under the sign conventions can be written as}
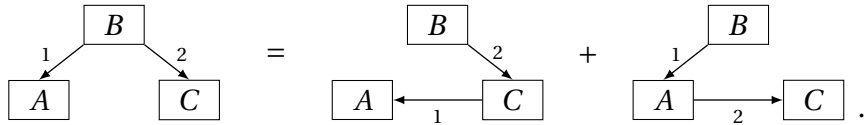
\begin{center}
    \begin{tikzpicture}
    \node (A) [rectangle, draw, minimum width=8mm] at (-1,0) {$A$};
    \node (B) [rectangle, draw, minimum width=8mm] at (0,1) {$B$};
    \node (C) [rectangle, draw, minimum width=8mm] at (1,0) {$C$};
    \draw[line width=0.5pt,latex-] (A.north)--(B.south west)node[midway, xshift=-2mm, yshift=1mm]{\tiny $1$}; 
    \draw[line width=0.5pt,-latex] (B.south east)--(C.north)node[midway,xshift=2mm,yshift=1mm]{\tiny $2$};
    \begin{scope}[xshift=4.3cm]
    \node (A2) [rectangle, draw, minimum width=8mm] at (-1,0) {$A$};
    \node (B2) [rectangle, draw, minimum width=8mm] at (0,1) {$B$};
    \node (C2) [rectangle, draw, minimum width=8mm] at (1,0) {$C$};
    \draw[line width=0.5pt,latex-] (A2.east)--(C2.west)node[midway,below]{\tiny $1$}; 
    \draw[line width=0.5pt,-latex] (B2.south east)--(C2.north)node[midway, xshift=1mm, yshift=1.25mm] {\tiny $2$};
    \end{scope}
    \begin{scope}[xshift=8.3cm]
    \node (A3) [rectangle, draw, minimum width=8mm] at (-1,0) {$A$};
    \node (B3) [rectangle, draw, minimum width=8mm] at (0,1) {$B$};
    \node (C3) [rectangle, draw, minimum width=8mm] at (1,0) {$C$};
    \draw[line width=0.5pt,latex-] (A3.north)--(B3.south west)node[midway,xshift=-1.25mm, yshift=1mm]{\tiny $1$};
    \draw[line width=0.5pt,-latex] (A3.east)--(C3.west)node[midway,below]{\tiny $2$};
    \end{scope}
    \coordinate (D) at ($(C.east)!0.5!(A2.west)$);
    \coordinate (E) at ($(C2.east)!0.5!(A3.west)$);
    \node at ([yshift=6mm]D) {$=$};
    \node at ([yshift=6mm]E) {$+$};
    \node at (0,1.5) {};
    \node at (0,-0.35) {};
    \end{tikzpicture}\enskip\raisebox{0.66em}{.}
    \end{center}
  \red{The point is that the term on the left hand-side encodes the linear cell given as the union of the two linear cells encoded by the summands on the right hand-side.}
    \red{To illustrate the phenomenon, consider the sum}    
    
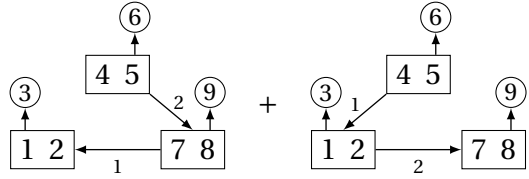
\begin{center}
    \begin{tikzpicture}
    \begin{scope}[xshift=4.3cm]
    \node (A2) [rectangle, draw, minimum width=8mm] at (-1,0) {$1 \enskip 2$};
    \node (r1) [circle, draw, inner sep=0.5pt,minimum width=4mm] at ([xshift=2mm,yshift=5mm]A2.north west) {\footnotesize $3$};
    \draw[line width=0.5pt,-latex] ([xshift=2mm]A2.north west)--(r1.south);
    \node (B2) [rectangle, draw, minimum width=8mm] at (0,1) {$4 \enskip 5$};
    \node (r2) [circle, draw, inner sep=0.5pt,minimum width=4mm] at ([xshift=-2mm,yshift=5mm]B2.north east) {\footnotesize $6$};
    \draw[line width=0.5pt,-latex] ([xshift=-2mm]B2.north east)--(r2.south);
    \node (C2) [rectangle, draw, minimum width=8mm] at (1,0) {$7 \enskip 8$};
    \node (r3) [circle, draw, inner sep=0.5pt,minimum width=4mm] at ([xshift=-2mm,yshift=5mm]C2.north east) {\footnotesize $9$};
    \draw[line width=0.5pt,-latex] ([xshift=-2mm]C2.north east)--(r3.south);
    \draw[line width=0.5pt,latex-] (A2.east)--(C2.west)node[midway,below]{\tiny $1$}; 
    \draw[line width=0.5pt,-latex] (B2.south east)--(C2.north)node[midway, xshift=1mm, yshift=1.25mm] {\tiny $2$};
    \end{scope}
    \begin{scope}[xshift=8.3cm]
    \node (A3) [rectangle, draw, minimum width=8mm] at (-1,0) {$1 \enskip 2$};
    \node (r4) [circle, draw, inner sep=0.5pt,minimum width=4mm] at ([xshift=2mm,yshift=5mm]A3.north west) {\footnotesize $3$};
    \draw[line width=0.5pt,-latex] ([xshift=2mm]A3.north west)--(r4.south);
    \node (B3) [rectangle, draw, minimum width=8mm] at (0,1) {$4 \enskip 5$};
    \node (r5) [circle, draw, inner sep=0.5pt,minimum width=4mm] at ([xshift=-2mm,yshift=5mm]B3.north east) {\footnotesize $6$};
    \draw[line width=0.5pt,-latex] ([xshift=-2mm]B3.north east)--(r5.south);
    \node (C3) [rectangle, draw, minimum width=8mm] at (1,0) {$7 \enskip 8$};
    \node (r6) [circle, draw, inner sep=0.5pt,minimum width=4mm] at ([xshift=-2mm,yshift=5mm]C3.north east) {\footnotesize $9$};
    \draw[line width=0.5pt,-latex] ([xshift=-2mm]C3.north east)--(r6.south);
    \draw[line width=0.5pt,latex-] (A3.north)--(B3.south west)node[midway,xshift=-1.25mm, yshift=1mm]{\tiny $1$};
    \draw[line width=0.5pt,-latex] (A3.east)--(C3.west)node[midway,below]{\tiny $2$};
    \end{scope}
    \coordinate (D) at ($(C.east)!0.5!(A2.west)$);
    \coordinate (E) at ($(C2.east)!0.5!(A3.west)$);
    \node at ([yshift=6mm]E) {$+$};
    \end{tikzpicture}
    \end{center} 
\red{which encodes the linear submanifold of $M_d^{(3)}(9)$ corresponding to the interior of} the union of the co-oriented \red{linear cells} $c_1$ and $c_2$ with common defining inequalities
\begin{equation}\label{comunes}
x_1^{(1)}=x_2^{(1)} \leq x_3^{(1)}, \;\;\;x_4^{(1)}=x_5^{(1)} \leq x_6^{(1)}\;\;\;\mbox{and}\;\;\; x_7^{(1)}=x_8^{(1)} \leq x_9^{(1)},
\end{equation}
together with the requirement that all $x_i$-coordinates have the same projection under $p_1$. The additional defining inequalities in $c_1$ are
\begin{equation}\label{adi1}
\underbrace{x_4^{(1)} \leq x_7^{(1)}}_2 \mbox{ \; and \;} \underbrace{x_7^{(1)}\leq x_1^{(1)}}_{1},
\end{equation}
while the additional defining inequalities in $c_2$ are
\begin{equation}\label{adi2}
\underbrace{x_4^{(1)} \leq x_1^{(1)}}_{1} \mbox{ \; and \;} \underbrace{x_1^{(1)}\leq x_7^{(1)}}_{2}.
\end{equation}
The union of conditions \red{(\ref{adi1}) and (\ref{adi2})} can \red{then} be stated as $$\underbrace{x_4^{(1)}\leq x_1^{(1)}}_{1} \mbox{ \;and \; } \underbrace{x_4^{(1)}\leq x_7^{(1)}}_2$$ which, together with (\ref{comunes}), \red{is encoded by}
\begin{center}
    \begin{tikzpicture}
    \node (A) [rectangle, draw, minimum width=8mm] at (-1,0) {$1 \enskip 2$};
    \node (r1) [circle, draw, inner sep=0.5pt,minimum width=4mm] at ([xshift=2mm,yshift=5mm]A.north west) {\footnotesize $3$};
    \draw[line width=0.5pt,-latex] ([xshift=2mm]A.north west)--(r1.south);
    \node (B) [rectangle, draw, minimum width=8mm] at (0,1) {$4 \enskip 5$};
    \node (r2) [circle, draw, inner sep=0.5pt,minimum width=4mm] at ([xshift=-2mm,yshift=5mm]B.north east) {\footnotesize $6$};
    \draw[line width=0.5pt,-latex] ([xshift=-2mm]B.north east)--(r2.south);
    \node (C) [rectangle, draw, minimum width=8mm] at (1,0) {$7 \enskip 8$};
    \node (r3) [circle, draw, inner sep=0.5pt,minimum width=4mm] at ([xshift=-2mm,yshift=5mm]C.north east) {\footnotesize $9$};
    \draw[line width=0.5pt,-latex] ([xshift=-2mm]C.north east)--(r3.south);
    \draw[line width=0.5pt,latex-] (A.north)--(B.south west)node[midway, xshift=-2mm, yshift=1mm]{\tiny $1$}; 
    \draw[line width=0.5pt,-latex] (B.south east)--(C.north)node[midway,xshift=2mm,yshift=1mm]{\tiny $2$};
    \end{tikzpicture}\enskip\raisebox{0.66em}{.}
    \end{center}

\begin{remark}\label{squarewithoutrounds}
\red{As detailed in \cite[Remark 6.2 and proof of Theorem 6.1]{dobri2015}, the} generalized Jacobi relation in 
Theorem \ref{sum} arises as the \red{Borel-Moore} boundary of a \red{linear} cell described by a \red{forest-like graph} one of whose square vertices has $k-2$ (rather than $k-1$) elements. \red{Now, by definition, the generalized Jacobi relation makes sense only for $\red{\omega}>1$. Yet, as observed by Dobrinskaya and Turchin, it is possible to consider oriented $k$-forests $T$ with square vertices admitting no neighbouring round vertices. For them, the corresponding linear cells $c_T$ are then Borel-Moore boundaries, thus validating the generalized Jacobi relation for $\omega=1$.}
\end{remark}

\section{\red{The cohomology of \texorpdfstring{$\mdkn$}{non-k-equal manifolds}: Products}}\label{secciondescripciondeproductos}
\red{The} cup product $T\red{{}\smile{}}T'$ of \red{oriented} $k$-forests \red{$T$ and $T'$ is} assessed geometrically \red{in \cite{dobri2015}} as the Poincar\'e dual of the fundamental class of the intersection $c_T\cap c_{T'}$. \red{We start by setting notation and basic} ingredients, \red{which can be found in standard references such as} \cite[Section V.11]{MR1481706}, \cite[Sections II.9, IX.3, IX.4 and IX.5]{MR842190}, \cite[Section 19.1]{MR732620} and \cite[Theorem~10.4]{Spanier93}.

For a locally compact space $Z$, there is a (sheaf theoretic supported) cap product $\frown\colon H^{\text{BM}}_a(Z)\otimes H^b(Z)\to H^{\text{BM}}_{a-b}(Z)$. This has several properties, including:
\begin{enumerate}
    \item $f_*(a'\frown f^*\xi)=f_*a'\frown\xi$, for any proper map $f\colon Z'\to Z$ and \red{arbitrary} classes $a'\in H^{\text{BM}}_*(Z')$ and $\xi\in H^*(Z)$.
    \item $(a\frown\xi)\frown\eta=a\frown(\xi\smile\eta)$, for \red{arbitrary} classes $\xi,\eta\in H^*(Z)$ and $a\in H^{\text{BM}}_*(Z)$.
    \item For an oriented $n$-dimensional (Hausdorff paracompact) manifold $N$, cap product with the fundamental class $[N]\in H^{\text{BM}}_n(N)$ yields a duality isomorphism $$D\colon H^*(N)\to H_{n-*}^{\text{BM}}(N).$$
    \item For an oriented properly embedded submanifold $V\subset N$ of codimension $\red{c}$, the orientation class $\mathfrak{o}^N_V\in H^{\red{c}}(N)$ of $V$ in $N$, i.e., the restriction of the (normal) Thom class $\mathfrak{u}^N_V\in H^{\red{c}}(N,N-V)$ of $V$ in $N$, \red{yields} $D(\mathfrak{o}^N_V)=[V]_N\red{{}\in H^{\text{BM}}_{n-c}(N),}$ the image of $[V]$ \red{under} the inclusion $V\hookrightarrow N$. 
\end{enumerate}
This information suffices to prove, just as in \cite[Theorem~\red{VI.}11.9]{Bredon1993}, that \red{cup products on a given oriented $n$-manifold $N$ can be assessed, in geometrical terms, through the \emph{intersection pairing}} at the bottom of the commutative square
\begin{equation}\label{cupvsinter}
\begin{gathered}
\begin{tikzpicture}
\node (O) at (-2,2) {$\otimes$};
\node[anchor=west] (B) at (O.east) {$H^{n-q}\left(N\right)$};
\node[anchor=east] (A) at (O.west) {$H^{n-p}\left(N\right)$};
\node[anchor=west] (C) at ([xshift=4cm]O.east) {$H^{2n-p-q}\left(N\right)$};
\draw[line width=1pt,-latex] (B.east)--(C.west)node[midway,above]{$\smile$};
\node (P) at (-2,0) {$\otimes$};
\node[anchor=west] (D) at (P.east) {$H_{\red{p}}^{\text{BM}}(N)$};
\node[anchor=east] (E) at (P.west) {$H_{\red{q}}^{\text{BM}}(N)$};
\node[anchor=west] (F) at ([xshift=4cm]P.east) {$H_{p+q-n}^{\text{BM}}(N)\red{.}$};
\draw[line width=1pt,-latex] (D.east)--(F.west)node[midway, above]{\!\!$\bullet$};
\draw[line width=1pt,-latex] ([yshift=1cm,xshift=12mm]E.north)--([xshift=12mm]E.north)node[midway,left]{$\mbox{twist}\circ(D\otimes D)$};
\draw[line width=1pt,-latex] ([yshift=1cm,xshift=0mm]F.north)--([xshift=0mm]F.north)node[midway,right]{$D$};
\end{tikzpicture}
\end{gathered}
\end{equation}

\begin{theorem}\label{interseccion} Let \red{the manifold} $N$ be as in item (3) above. If \red{the oriented submanifolds} $X$ and $Y$ are properly embedded in $N$ \red{and have} transverse intersection, then
\[
[X]_N \bullet [Y]_N = [X\cap Y]_N.
\]
\end{theorem}

\red{The use of Theorem~\ref{interseccion} in the case of non-$k$-equal manifolds leads to:}
\begin{definition}
\red{Two oriented $k$-forests} $T_1$, $T_2 \in H^\ast(\mdkn)$
\red{are said to be \emph{superposable} when} no square vertex of $T_1$ intersects a square vertex of\espacio$T_2$. \red{In such a case} we define the \red{\emph{\blue{bent} superposition}} $T_1\cup T_2$ as the oriented \red{$\nn$-bipartitioned} graph obtained by superposition of the \red{oriented $\nn$-bipartitioned graphs underlying} $T_1$ and\espacio$T_2$. \red{This means that  square vertices, their contents, and oriented edges between such vertices in $T_1\cup T_2$ are those holding either on $T_1$ or on $T_2$. Likewise, round vertices with content $i$, as well as oriented edges involving such vertices in $T_1\cup T_2$ are those holding either on $T_1$ or on $T_2$, \emph{as long as} $i$ has not been accounted for by some square vertex. Instead,} if some integer $i \in \nn$ lies in a square vertex $A$ in, say, $T_{1}$ as well as in a round vertex attached to some square vertex $B$ through an oriented edge in, respectively, $T_{2}$, then not only $i$ appears in $T_1\cup T_2$ inside the corresponding square vertex\espacio$A$, but a corresponding oriented \red{``\blue{bent}''} edge in $T_1\cup T_2$ between $A$ and $B$ \red{must be added. The left and right-hand sides in Figure \ref{bendededge} sketch the relevant situations for $T_2$ and $T_1\cup T_2$, respectively.}
\begin{figure}[ht] 
\centering
\begin{tikzpicture}[scale=0.85]
\node (A) [circle, draw, inner sep=0.5pt,minimum width=5mm] {$i$};
\node (C) [rectangle, draw, minimum width=1.6cm, minimum height=0.7cm] at (-0.3,0) {\hspace*{-8mm}$A$};
\node (B) [rectangle, draw, minimum width=1.4cm, minimum height=0.7cm] at (3,0) {$B$};
\draw[line width=0.5pt,-latex] (B.north) to[out=130,in=50] (C.north);
\node (D) [rectangle, draw, minimum width=1.4cm, minimum height=0.7cm] at (-5,0) {$B$};
\node (E) [circle, draw, inner sep=0.5pt,minimum width=5mm] at (-5,1.5) {$i$};
\draw[line width=0.5pt,-latex] (D.north) to (E.south);
\end{tikzpicture}
\caption{Bending an edge}
\label{bendededge}
\end{figure}
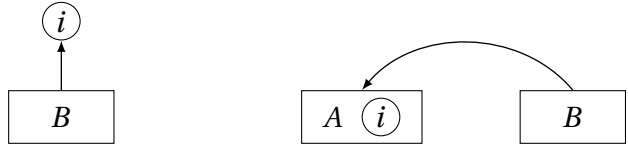
\red{Note that $T_1\cup T_2$ might end up having} multiple oriented edges between square vertices, as well as round vertices having two square vertices as immediate neighboring vertices.
\end{definition}

\begin{theorem}[{\cite[Theorem 7.1]{dobri2015}}]\label{prod} \red{The cup product $T_1\smile T_2$ of two oriented $k$-forests $T_1$, $T_2 \in H^\ast(\mdkn)$  vanishes in either of the} following three conditions:
\begin{itemize}
\item[\red{(A)}] \red{$T_1$ and $T_2$ are not superposable.} 
\item[\red{(B)}] \red{$T_1$ and $T_2$ are superposable and} $T_1 \cup T_2$ has unoriented cycles (for instance if two square vertices of $T_1 \cup T_2$ are joined by multiple edges). 
\item[\red{(C)}] \red{$T_1$ and $T_2$ are superposable and} $T_1 \cup T_2$ has a square vertex with no round vertex attached.
\end{itemize}
\red{In any other case the intersection $c_{T_1}\cap c_{T_2}$ is transverse and} $$T_1 \red{{}\smile{}} T_2\red{{}={}}T_1\cup T_2$$ with orientation set given by the concatenation of the orientation sets of the factors, and with the convention that, if $T_1\cup T_2$ \red{happens not to be} a $k$-forest (in the sense of Definition~\ref{defkforesto}), so that $T_1 \cup T_2$ has one or several round vertices of valency $2$, then we \red{transform $T_1 \cup T_2$ into a sum of oriented $k$-forests through repeated use of orientation relations and} the following form of the three-term relation:
\begin{equation}\label{productrelation} 
\begin{tikzpicture}[baseline=(current bounding box.center)]
\node (A) [rectangle, draw, minimum width=8mm] at (-1,0) {$A$};
\node (B) [circle, draw, minimum width=5mm] at (0,1) {};
\node (C) [rectangle, draw, minimum width=8mm] at (1,0) {$B$};
\draw[line width=0.5pt,latex-] (A.north)--(B.210)node[midway,xshift=-2mm,yshift=1mm]{\tiny $1$};
\draw[line width=0.5pt,-latex] (B.330)--(C.north)node[midway,xshift=2mm, yshift=1mm]{\tiny $2$};
\begin{scope}[xshift=3.75cm]
\node (A2) [rectangle, draw, minimum width=8mm] at (-1,0) {$A$};
\node (B2) [circle, draw, minimum width=5mm] at (0,1) {};
\node (C2) [rectangle, draw, minimum width=8mm] at (1,0) {$B$};
\draw[line width=0.5pt,-latex] (A2.east)--(C2.west)node[midway,below]{\tiny $2$};
\draw[line width=0.5pt,-latex] (B2.210)--(A2.north)node[midway,xshift=-2mm,yshift=1mm]{\tiny $1$};
\end{scope}
\begin{scope}[xshift=7.5cm]
\node (A3) [rectangle, draw, minimum width=8mm] at (-1,0) {$A$};
\node (B3) [circle, draw, minimum width=5mm] at (0,1) {};
\node (C3) [rectangle, draw, minimum width=8mm] at (1,0) {$B$};
\draw[line width=0.5pt,latex-] (C3.north)--(B3.330)node[midway,xshift=2mm,yshift=1mm]{\tiny $2$};
\draw[line width=0.5pt,latex-] (A3.east)--(C3.west)node[midway,below]{\tiny $1$};
\end{scope}
\coordinate (D) at ($(C.east)!0.5!(A2.west)$);
\coordinate (E) at ($(C2.east)!0.5!(A3.west)$);
\node at ([yshift=6mm]D) {$=$};
\node at ([yshift=6mm]E) {$+$};
\end{tikzpicture}\enskip\raisebox{-1.85em}{.}
\end{equation}
As above, pictures are local.
\end{theorem}

Items \red{(B)} and \red{(C)} in Theorem~\ref{prod} might have to be used in the iterative process of applying relation~(\ref{productrelation}) \red{in order} to write \red{a non-$k$-forest} $T_1\cup T_2$ as a sum of \red{oriented} $k$-forests. For instance, if the pictures in~(\ref{productrelation}) are in fact global (omitting \red{possible} isolated round vertices), then \red{each of} the two summands on the right of~(\ref{productrelation}) would vanish in view of item~(C) in Theorem~\ref{prod}.

Relevant for us is the fact that $H^\ast(\mdkn)$ is multiplicatively generated by \red{\emph{elementary} $k$-forests, i.e., basic $k$-forests} having a single square vertex. Explicitly, a basic $k$-forest is, up to sign, the product of its connected components. In turn, each such connected component is, up to sign, a product of elementary $k$-forests. For example, the basic $3$-forest
$$
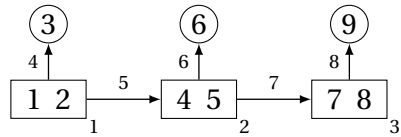
\begin{tikzpicture}[baseline=(current bounding box.center)]
\node (A) [rectangle, draw, minimum width=10mm] at (-1,0) {$1\enskip 2$};
\node (B) [rectangle, draw, minimum width=10mm] at (1,0) {$4\enskip 5$};
\node (C) [rectangle, draw, minimum width=10mm] at (3,0) {$7\enskip 8$};
\node (r1) [circle, draw, inner sep=0.5pt,minimum width=5mm] at (-1,1) {$3$};
\node (r2) [circle, draw, inner sep=0.5pt,minimum width=5mm] at (1,1) {$6$};
\node (r3) [circle, draw, inner sep=0.5pt,minimum width=5mm] at (3,1) {$9$};
\node at ([xshift=1mm,yshift=-1mm]A.south east) {\tiny $1$};
\node at ([xshift=1mm,yshift=-1mm]B.south east) {\tiny $2$};
\node at ([xshift=1mm,yshift=-1mm]C.south east) {\tiny $3$};
\draw[line width=0.5pt,-latex] (A.north)--(r1.south)node[midway,left]{\tiny $4$};
\draw[line width=0.5pt,-latex] (B.north)--(r2.south)node[midway,left]{\tiny $6$};
\draw[line width=0.5pt,-latex] (C.north)--(r3.south)node[midway,left]{\tiny $8$};
\draw[line width=0.5pt,-latex] (A.east)--(B.west)node[midway,above]{\tiny $5$};
\draw[line width=0.5pt,-latex] (B.east)--(C.west)node[midway,above]{\tiny $7$};
\end{tikzpicture}
$$
\blue{can be factorized as}
\[
\left(\rule[-6mm]{0pt}{12mm}\right.\  
\begin{tikzpicture}[baseline=(current bounding box.center)]
\node (A) [rectangle, draw, minimum width=10mm] at (-1,0) {$1\enskip 2$};
\node (r1) [circle, draw, inner sep=0.5pt,minimum width=5mm] at (-1.35,1) {$3$};
\node (r2) [circle, draw, inner sep=0.5pt,minimum width=5mm] at (-0.65,1) {$4$};
\node at ([xshift=1mm,yshift=-1mm]A.south east) {\tiny $1$};
\draw[line width=0.5pt,-latex] ([xshift=-1mm]A.north)--(r1.south)node[midway,left]{\tiny $2$};
\draw[line width=0.5pt,-latex] ([xshift=1mm]A.north)--(r2.south)node[midway,right]{\tiny $3$};
\end{tikzpicture}\left.\rule[-6mm]{0pt}{12mm}\right)
\left(\rule[-6mm]{0pt}{12mm}\right.\  
\begin{tikzpicture}[baseline=(current bounding box.center)]
\node (A) [rectangle, draw, minimum width=10mm] at (-1,0) {$4\enskip 5$};
\node (r1) [circle, draw, inner sep=0.5pt,minimum width=5mm] at (-1.35,1) {$6$};
\node (r2) [circle, draw, inner sep=0.5pt,minimum width=5mm] at (-0.65,1) {$\red{8}$};
\node at ([xshift=1mm,yshift=-1mm]A.south east) {\tiny $1$};
\draw[line width=0.5pt,-latex] ([xshift=-1mm]A.north)--(r1.south)node[midway,left]{\tiny $2$};
\draw[line width=0.5pt,-latex] ([xshift=1mm]A.north)--(r2.south)node[midway,right]{\tiny $3$};
\end{tikzpicture}\left.\rule[-6mm]{0pt}{12mm}\right)
\left(\rule[-6mm]{0pt}{12mm}\right.\  
\begin{tikzpicture}[baseline=(current bounding box.center)]
\node (A) [rectangle, draw, minimum width=10mm] at (-1,0) {$7\enskip 8$};
\node (r2) [circle, draw, inner sep=0.5pt,minimum width=5mm] at (-1,1) {$9$};
\node at ([xshift=1mm,yshift=-1mm]A.south east) {\tiny $1$};
\draw[line width=0.5pt,-latex] ([xshift=0mm]A.north)--(r2.south)node[midway,right]{\tiny $2$};
\end{tikzpicture}\left.\rule[-6mm]{0pt}{12mm}\right)\red{.}
\vspace{-1.2mm}
\]
\red{Note that factorizations are not unique.}

\red{The following additional piece of information regarding items (A)--(C) in Theorem\espacio\ref{prod} will be useful in the next section.}
\begin{lemma}\label{posiblesproductos}
\red{Consider cohomology classes $u,v\in H^*(\mdkn)$ represented, respectively, by elementary $k$-forests} 
\[
\mbox{$\red{T_u}=$
    \begin{tikzpicture} [baseline] 
    \node (r1) [circle, draw, minimum width=1mm,inner sep=1pt] at (-0.75,0.85) {$\red{b}_1$};
    \node (r2) [circle, draw, minimum width=1mm,inner sep=1pt] at (0.75,0.85) {$\red{b_r}$};
    \node (A1) [rectangle, draw, minimum width=1cm] at (0,0) {\red{$a_1 \cdots a_{k-1}$}};
    \draw[line width=0.5pt] (A1.north -| r1.south)--(r1.south);
    \node at (0,0.85) {$\cdots$};
    \draw[line width=0.5pt] (A2.north -| r2.south)--(r2.south);
    \end{tikzpicture}
    \quad \text{and} \quad 
    $\red{T_v}=$
    \begin{tikzpicture} [baseline] 
    \node (r1) [circle, draw, minimum width=1mm,inner sep=1pt] at (-0.75,0.85) {$\red{d}_1$};
    \node (r2) [circle, draw, minimum width=1mm,inner sep=1pt] at (0.75,0.85) {$\red{d_s}$};
    \node (A1) [rectangle, draw, minimum width=1cm] at (0,0) {$\red{c_1\cdots c_{k-1}}$};
    \node at (1.1,-.15) {\red{.}};
    \draw[line width=0.5pt] (A1.north -| r1.south)--(r1.south);
    \node at (0,0.85) {$\cdots$};
    \draw[line width=0.5pt] (A1.north -| r2.south)--(r2.south);
    \end{tikzpicture}}   
\]
\begin{itemize}
\item[\emph{(a)}]\red{If $T_u$ and $T_v$ are not superposable (so that $u\smile v=0$), then representing $k$-forests and linear cells can be chosen so that $c_{T_u} \cap c_{T_v}$ is empty in $M^{(k)}_d(n)$.}
\item[\emph{(b)}] \red{Assume $T_u$ and $T_v$ are superposable, still with $u\smile v=0$, and set
$\omega:=\card\,(\left\{a_1,\ldots,a_{k-1},b_1,\ldots, b_r \right\} \cap \left\{c_1,\ldots,c_{k-1},d_1,\ldots,d_s\right\})$ ---so that $\omega>0$. Then one of the following options must hold:}
\begin{itemize}
    \item[{\emph{(b.1)}}] $\,\,\,\red{\omega}>1$;
    \item[{\emph{(b.2)}}] $\,\,\,\red{\omega}=\red{r}=1$ \red{and} $\red{b}_1\in \red{\{c_1,\ldots,c_{k-1}\}}$;
    \item[{\emph{(b.3)}}] $\,\,\,\red{\omega}=\red{s}=1$ \red{and} $\red{d}_1\in \red{\{a_1,\ldots,a_{k-1}\}}$;
    \item[{\emph{(b.4)}}] \red{$\,\,\,\red{\omega}=r=s=1$ and $b_1=d_1$.}
\end{itemize}
\end{itemize}
\red{Furthermore, the conclusion in (a) also holds true in case (b.1), whereas the intersection $c_U\cap c_V$ is a Borel-Moore boundary in cases (b.2)--(b.4).}
\end{lemma}
\begin{proof}
\red{The fact that, under the hypotheses in \emph{(b)}, one of \emph{(b.1)--(b.4)} must hold follows from Theorems\espacio\ref{sum} and \ref{prod}, in view of the trivial-cup-product hypothesis. The assertion about $c_{T_u}\cap c_{T_v}$ in \emph{(b.2)--(b.4)} comes from Remark \ref{squarewithoutrounds}. Lastly, the empty-intersection condition in the case of non-superposable factors, as well as for $\omega>1$ with superposable factors, is contained in the proof of Theorem 7.1 in \cite{dobri2015}. As indicated by Dobrinskaya and Truchin, this might require a small adjustment of representing elementary $k$-forests, at the cohomology level, or corresponding linear cells, at the homology level. Explicitly, \blue{in both cases, the codimension of $c_{T_u}\cap c_{T_v}$ is not the sum of the factor's codimensions, showing that the intersection is not transverse. Hence, to assess the intersection product of $c_{T_u}$ and $c_{T_v}$} 
an application of the General Position Lemma is needed to shift slightly one of the linear cells \blue{to obtain a new linear cell, let's say $c_{T_v}'$, such that the intersection $c_{T_u}\cap c_{T_v}'$ now is vacuously transverse ---because the slight perturbation of the $w$ equal variables would make the linear cells $c_{T_u}$ and $c_{T_v}'$ disjoint. This slight perturbation has no effect on the (co)homology classes involved.}}  

\end{proof}

\section{Massey products and duality}\label{mps}
\red{For the rest of the paper we deal exclusively with non-3-equal manifolds $M^{(3)}_d(n)$. \blue{O}rientation issues will be neglected by working with mod-$2$ coefficients \blue{to simplify arguments}. In particular, edges of \blue{a} $3$-forests $T$ will no longer be oriented. Yet, we will keep the convention that linear cells $c_T$, and the corresponding submanifolds $\text{Int}(c_T)$ of $M^{(3)}_d(n)$ encoded by $T$, are taken (in terms of (\ref{lado})) as if edge orientations were the canonical ones in Figure~\ref{semilinear}, i.e., assuming that edges point right or upwards. See for instance (\ref{ori1}) and (\ref{ori2}) below.}

In his seminal work \cite{Massey1968}, Massey introduced a geometric method to \red{assess} his higher order cohomology \red{products.} \red{The idea is to use Poincar\'e duality in order} to replace cup products \red{at the cochain level by intersections of dual submanifolds at the (locally compact) chain level. Variants} of the technique have been used in knot theory to compute higher order linking numbers (\cite{MR790675,MR3672260,MR549154}). \red{Intersection theory has also} been used to evaluate Massey products \red{on} classical configuration spaces (\cite{MR2114713,MR2846157}). \red{With mod-2 coefficients,} the basic (folklore) observation \red{is summarized in Remark\espacio\ref{folklore} below, where we use the symbol $\pitchfork$ to indicate that an intersection of submanifolds is transverse.}

\begin{remark}\label{folklore}
\red{Consider properly embedded} submanifolds $K$, $L$ and $M$ of \red{some manifold} $N$, \red{and} let $\kappa$, $\lambda$ and $\mu$ \red{denote} the Poincar\'e duals of the fundamental classes
$[K]_N,[L]_N,[M]_N\in H_*^{\text{BM}}(N)$. Assume $K\pitchfork L=\partial X$ and $L\pitchfork M=\partial Y$ for submanifolds $X$ and $Y$ with $X\pitchfork M$ and $K\pitchfork Y$. Then the triple Massey product $\left\langle \kappa,\lambda,\mu\right\rangle$ contains the Poincar\'e dual of the fundamental class $\big[\red{(}X\cap M\red{)} \red{{}\cup{}} \red{(}K\cap Y\red{)}\big]_N\,.$
\end{remark}

\red{In the computations below, $M^{(3)}_d(n)$ will play the role of $N$, while} the submanifolds \red{$L$, $K$ and $M$ will be encoded by suitably chosen} $\red{3}$-forests. \red{On the other hand, the} submanifolds $X$ and\espacio$Y$ \red{will fail to correspond to honest 3-forests, but will be encoded through a similar terminology (cf. Remark \ref{squarewithoutrounds}). For example, the first forest-like graph in
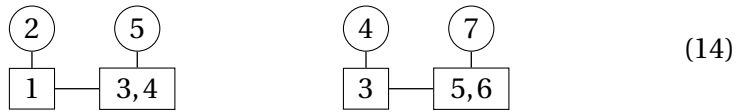
\begin{equation}\label{subXY}
\begin{tikzpicture}[baseline=(current bounding box.center)]
\node (r1) [rectangle, draw, minimum width=6mm,inner sep=4pt] at (0,0) {$1$};
\node (r2) [circle, draw, minimum width=6mm,inner sep=1pt] at (0,0.8) {$2$};
\node (r3) [circle, draw, minimum width=6mm,inner sep=1pt] at (1.4,0.8) {$5$};
\node (A) [rectangle, draw, minimum width=1cm, inner sep=0pt] at (1.4,0) {$3,4$\rule[-1.5mm]{0pt}{5.5mm}};
\draw[line width=0.5pt] (r2)--(r1)--(A)--(r3);
\end{tikzpicture}
\qquad\qquad\qquad
\begin{tikzpicture}[baseline=(current bounding box.center)]
\node (r1) [rectangle, draw, minimum width=6mm,inner sep=4pt] at (0,0) {$3$};
\node (r2) [circle, draw, minimum width=6mm,inner sep=1pt] at (0,0.8) {$4$};
\node (r3) [circle, draw, minimum width=6mm,inner sep=1pt] at (1.4,0.8) {$7$};
\node (A) [rectangle, draw, minimum width=1cm, inner sep=0pt] at (1.4,0) {$5,6$\rule[-1.5mm]{0pt}{5.5mm}};
\draw[line width=0.5pt] (r2)--(r1)--(A)--(r3);
\end{tikzpicture}
\end{equation}
encodes the $(4d-3)$-codimensional linear submanifold of $M^{(3)}_d(n)$ determined by the conditions
\begin{align}
&p_1(x_i)=p_1(x_j), \text{ \,\, for } i,j\in\{1,2,3,4,5\}, \nonumber\\
&x_3=x_4, \label{ori1}\\
&x^{(1)}_1\leq x^{(1)}_2,\,\,x^{(1)}_1\leq x^{(1)}_3,\,\,x^{(1)} _3\leq x^{(1)}_5.\nonumber
\end{align}
Likewise, the forest-like graph on the right hand-side of (\ref{subXY}) encodes the $(4d-3)$-codimensional linear submanifold of $M^{(3)}_d(n)$ determined by the conditions}
\begin{align}
&\red{p_1(x_i)=p_1(x_j), \text{ \,\, for } i,j\in\{3,4,5,6,7\},} \nonumber\\
&\red{x_5=x_6,} \label{ori2}\\
&\red{x^{(1)}_3\leq x^{(1)}_4,\,\,x^{(1)}_3\leq x^{(1)}_5,\,\,x^{(1)} _5\leq x^{(1)}_7.}\nonumber
\end{align}

\begin{remark}\label{cuadrote}
\red{As cohomology} classes, the \red{3}-forests
$$
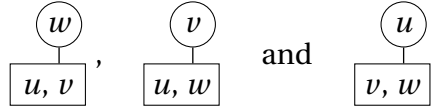
\elementary{u}{v}{w},\quad\elementary{u}{w}{v}\quad\text{\,and\,}\quad \elementary{v}{w}{u}
$$
\red{agree, in view of Theorem \ref{sum}(3). The common cohomology class will simply} be denoted by
$$
    
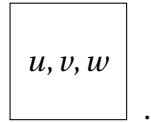
\begin{tikzpicture}[baseline=7.2pt]
    \node (A) [minimum size=1.55cm, draw, minimum width=1.55cm, inner sep=0pt] at (1.4,0.23) {\red{$u,v,w$}\rule[-1.5mm]{0pt}{5.5mm}};
    \node at (2.45,-.51) {.};
    \end{tikzpicture}
$$
\end{remark}


\red{Part of the subtleties in the search of non-trivial Massey products in the first meaningful case outside the range in\espacio(\ref{loworder}) comes from:}
\begin{proposition}\label{ternarioscero} Every \red{well-defined} triple Massey product of elementary \red{basis elements in $H^*(M^{(3)}_d(n))$ vanishes (modulo indeterminacy).}
\end{proposition}
\begin{proof} \red{Consider a triple Massey product $\langle \kappa, \lambda, \mu\rangle$ of classes represented by elementary 3-forests} $\kappa$, $\lambda$, and $\mu$ \red{(so that $\kappa\smile\lambda=0=\lambda\smile\mu$).} It suffices to argue that, \red{after perhaps a slight adjustment of representing elementary $3$-forests or linear cells,}
\begin{equation}\label{biendefinido}
c_\kappa\cap c_\lambda\text{ and }c_\lambda\cap c_\mu\text{ are empty in $M^{(3)}_d(n)$,}
\end{equation}
\red{for then $0\in \langle \kappa, \lambda, \mu\rangle$, by Remark \ref{folklore}. In view of Lemma~\ref{posiblesproductos}, the only way in which (\ref{biendefinido}) could fail is if some of the representing elementary 3-forests have a single round vertex attached to its single square vertex. But in those cases, Remark\espacio\ref{cuadrote} can be used to assure the required conditions forcing (\ref{biendefinido}).}
\end{proof}

\red{As a warmup for the proof that $M^{(3)}_d(n)$ supports non-trivial triple Massey products for $n>6$, we illustrate the arguments in a simpler situation.}
\begin{example}\label{ejemplomassey0} \red{Since the product}
\begin{equation}\label{theproduct}
\left(\begin{tikzpicture}[baseline=3mm]
    \node (r1) [circle, draw, minimum width=6mm,inner sep=1pt] at (-0.7,0.8) {$2$};
    \node (r2) [circle, draw, minimum width=6mm,inner sep=1pt] at (0,0.8) {$4$};
    \node (r3) [circle, draw, minimum width=6mm,inner sep=1pt] at (0.7,0.8) {$5$};
    \node (A) [rectangle, draw, minimum width=1.5cm, inner sep=0pt] at (0,0) {$1,3$\rule[-1.5mm]{0pt}{5.5mm}};
    \draw[line width=0.5pt] (r1.south)--(r1|-A.north) (r2.south)--(r2|-A.north) 
    (r3.south)--(r3|-A.north);
    \end{tikzpicture}\right)\ \left(
\elementary{5}{6}{7}
\right)=\raisebox{1.3mm}{
\begin{tikzpicture}[baseline=(current bounding box.center)]
    \node (r4) [circle, draw, minimum width=6mm,inner sep=1pt] at (-0.35,0.8) {$2$};
    \node (r5) [circle, draw, minimum width=6mm,inner sep=1pt] at (0.35,0.8) {$4$};
    \node (r1) [rectangle, draw, minimum width=1cm, inner sep=0pt] at (0,0) {$1,3$\rule[-1.5mm]{0pt}{5.5mm}};
    \node (r3) [circle, draw, minimum width=6mm,inner sep=1pt] at (1.4,0.8) {$7$};
    \node (A) [rectangle, draw, minimum width=1cm, inner sep=0pt] at (1.4,0) {$5,6$\rule[-1.5mm]{0pt}{5.5mm}};
    \draw[line width=0.5pt] (r5.south)--(r5|-r1.north) (r4.south)-- (r4|-r1.north) (r1)--(A)--(r3);
    \end{tikzpicture}}
\end{equation}
    lies in the indeterminacy of
\begin{equation}\label{tmpexe}
\left\langle
\elementary{1}{2}{3},\elementary{3}{4}{5},\elementary{5}{6}{7}
\right\rangle,
\end{equation}
\red{Proposition~\ref{ternarioscero} implies that (\ref{theproduct}) lies in (\ref{tmpexe}). Our goal is to recover} the latter assertion \red{through a direct geometric argument. Consider the submanifolds $K$, $L$ and $M$ of $N:=M^{(3)}_d(n)$ encoded, respectively, by}
\[
\elementary{1}{2}{3},\,\enskip 
\elementary{3}{4}{5}\,\enskip\mbox{and}\enskip
\elementary{5}{6}{7}\red{.}
\]
\red{The submanifolds $X$ and $Y$ encoded, respectively, by the forest-like diagrams in (\ref{subXY}) satisfy $K\pitchfork L=\partial X$ and $L\pitchfork M=\partial Y$ in $N$. Indeed, no boundary condition (i.e., an equality that replaces an inequality) can be taken with respect to the $\leq$-inequalities coming from the edges of the square vertices in (\ref{subXY}) containing $\{3,4\}$ or $\{5,6\}$, for otherwise we \blue{fall outside} $N$. Furthermore $X\pitchfork M$ and $K\pitchfork Y$ are respectively encoded by}
$$
\begin{tikzpicture}[baseline=(current bounding box.center)]
    \node (r4) [rectangle, draw, minimum width=6mm,inner sep=4pt] at (-1.4,0) {$1$};
    \node (r5) [circle, draw, minimum width=6mm,inner sep=1pt] at (-1.4,0.8) {$2$};
    \node (r1) [rectangle, draw, minimum width=1cm, inner sep=0pt] at (0,0) {$3,4$\rule[-1.5mm]{0pt}{5.5mm}};
    \node (r3) [circle, draw, minimum width=6mm,inner sep=1pt] at (1.4,0.8) {$7$};
    \node (A) [rectangle, draw, minimum width=1cm, inner sep=0pt] at (1.4,0) {$5,6$\rule[-1.5mm]{0pt}{5.5mm}};
    \draw[line width=0.5pt] (r5)--(r4)--(r1)--(A)--(r3);
    \end{tikzpicture}\qquad\raisebox{-5mm}{and}\qquad
\begin{tikzpicture}[baseline=(current bounding box.center)]
    \node (B) [rectangle, draw, minimum width=1cm, inner sep=0pt] at (-1.4,0) {$1,2$\rule[-1.5mm]{0pt}{5.5mm}};
    \node (r1) [rectangle, draw, minimum width=6mm,inner sep=4pt] at (0,0) {$3$};
    \node (r2) [circle, draw, minimum width=6mm,inner sep=1pt] at (0,0.8) {$4$};
    \node (r3) [circle, draw, minimum width=6mm,inner sep=1pt] at (1.4,0.8) {$7$};
    \node (A) [rectangle, draw, minimum width=1cm, inner sep=0pt] at (1.4,0) {$5,6$\rule[-1.5mm]{0pt}{5.5mm}};
    \draw[line width=0.5pt] (r2)--(r1)--(A)--(r3) (B)--(r1);
\end{tikzpicture}\,\,\,\raisebox{-14pt}{\red{.}}
$$
\red{The goal is then achieved in view of Remark \ref{folklore}, as the fundamental class $\varphi:=[(X\cap M) \cup (K\cap Y)]_N\in H^{\text{BM}}_*(M^{(3)}_d(n))$ is Poincar\'e dual of (\ref{theproduct}). Indeed, the union $(X\cap M) \cup (K\cap Y)$ sits inside the boundary of the submanifold $Z$ of $N$ encoded by}
\[
\begin{tikzpicture}[baseline=(current bounding box.center)]
\node (r4) [rectangle, draw, minimum width=6mm,inner sep=4pt] at (-1.4,0) {$1$};
\node (r5) [circle, draw, minimum width=6mm,inner sep=1pt] at (-1.4,0.8) {$2$};
\node (r1) [rectangle, draw, minimum width=6mm,inner sep=4pt] at (0,0) {$3$};
\node (r2) [circle, draw, minimum width=6mm,inner sep=1pt] at (0,0.8) {$4$};
\node (r3) [circle, draw, minimum width=6mm,inner sep=1pt] at (1.4,0.8) {$7$};
\node (A) [rectangle, draw, minimum width=1cm, inner sep=0pt] at (1.4,0) {$5,6$\rule[-1.5mm]{0pt}{5.5mm}};
\draw[line width=0.5pt] (r2)--(r1)--(A)--(r3) (r5)--(r4)--(r1);
\end{tikzpicture}\,\,\,\red{\raisebox{-6mm}{$,$}}
\]
\red{while the} rest of the boundary \red{of $Z$ is a manifold representing $\varphi$ and clearly encoded by (\ref{theproduct}).}
\end{example}

The verification of the following \red{auxiliary} fact is an elementary exercise using the cup-product description\red{s} in Section~\ref{secciondescripciondeproductos}.

\begin{lemma}\label{inprepa}
The product of two \red{elementary} \blue{terms}
$$ 
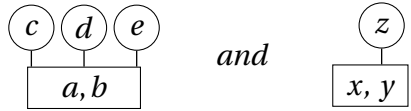
\begin{tikzpicture}[baseline=3mm]
    \node (r1) [circle, draw, minimum width=6mm,inner sep=1pt] at (-0.7,0.8) {$c$};
    \node (r2) [circle, draw, minimum width=6mm,inner sep=1pt] at (0,0.8) {$d$};
    \node (r3) [circle, draw, minimum width=6mm,inner sep=1pt] at (0.7,0.8) {$e$};
    \node (A) [rectangle, draw, minimum width=1.5cm, inner sep=0pt] at (0,0) {$a,b$\rule[-1.5mm]{0pt}{5.5mm}};
    \draw[line width=0.5pt] (r1.south)--(r1|-A.north) (r2.south)--(r2|-A.north) (r3.south)--(r3|-A.north);
\end{tikzpicture} \qquad \text{and}\qquad
\elementary{x}{y}{z}
$$
in $H^*(M^{(3)}_d(n))$ is either zero or a sum of \red{basic} elements of one of the two forms
$$
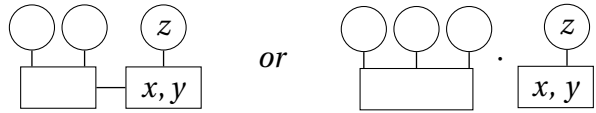
\begin{tikzpicture}[baseline=2.8mm]
    \node (r4) [circle, draw, minimum width=6mm,inner sep=1pt] at (-0.35,0.8) {};
    \node (r5) [circle, draw, minimum width=6mm,inner sep=1pt] at (0.35,0.8) {};
    \node (r1) [rectangle, draw, minimum width=1cm, inner sep=0pt] at (0,0) {\rule[-1.5mm]{0pt}{5.5mm}};
    \node (r3) [circle, draw, minimum width=6mm,inner sep=1pt] at (1.4,0.8) {$z$};
    \node (A) [rectangle, draw, minimum width=1cm, inner sep=0pt] at (1.4,0) {$x,y$\rule[-1.5mm]{0pt}{5.5mm}};
    \draw[line width=0.5pt] (r5.south)--(r5|-r1.north) (r4.south)-- (r4|-r1.north) (r1)--(A)--(r3);
    \end{tikzpicture}
    \qquad \text{or}\qquad 
    \begin{tikzpicture}[baseline=3mm]
    \node (r1) [circle, draw, minimum width=6mm,inner sep=1pt] at (-0.7,0.8) {};
    \node (r2) [circle, draw, minimum width=6mm,inner sep=1pt] at (0,0.8) {};
    \node (r3) [circle, draw, minimum width=6mm,inner sep=1pt] at (0.7,0.8) {};
    \node (A) [rectangle, draw, minimum width=1.5cm, inner sep=0pt] at (0,0) {\rule[-1.5mm]{0pt}{5.5mm}};
    \draw[line width=0.5pt] (r1.south)--(r1|-A.north) (r2.south)--(r2|-A.north) 
    (r3.south)--(r3|-A.north);
\end{tikzpicture}\cdot \elementary{x}{y}{z}
$$
and, in either case, the set of numbers inside the vertices of each of these summands is precisely $\{a,b,c,d,e,x,y,z\}$.
\end{lemma}

\begin{theorem}\label{ptmnt} For $n>6$, the \red{triple} Massey product in $M_d^{(3)}(n)$
\begin{equation}\label{notrivial}
 \left\langle
\elementary{1}{2}{3},\elementary{3}{4}{5},\elementary{5}{6}{7}{+}\elementary{4}{6}{7}{+}\elementary{4}{5}{7}{+}\elementary{4}{5}{6}
\right\rangle
\end{equation}
is \red{well-defined and non-trivial.}
\end{theorem}
\begin{proof} \red{Well-definedness is obvious. Consider the submanifolds $K$, $L$, $M$, $X$ and $Y$ in Example \ref{ejemplomassey0}, together with the submanifolds $M_5$, $M_6$ and $M_7$ encoded, respectively, by the second, third and fourth summands of}
\begin{align*}
\elementary{5}{6}{7}+\elementary{4}{6}{7}+\elementary{4}{5}{7}+\elementary{4}{5}{6}\raisebox{-4mm}{\red{.}}
\end{align*}
\red{Note that $L\cap M_i=\varnothing$ for $5\leq i\leq7$, so that $L\pitchfork\widetilde{M}=L\pitchfork M=\partial Y$, where $\widetilde{M}=M\cup M_5\cup M_6\cup M_7$. Since $X\cap M_i=\varnothing$ for $5\leq i\leq7$, the conclusion in Example \ref{ejemplomassey0} extends to yield} that 
$$ 
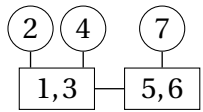
\begin{tikzpicture}[baseline=(current bounding box.center)]
    \node (r4) [circle, draw, minimum width=6mm,inner sep=1pt] at (-0.35,0.8) {$2$};
    \node (r5) [circle, draw, minimum width=6mm,inner sep=1pt] at (0.35,0.8) {$4$};
    \node (r1) [rectangle, draw, minimum width=1cm, inner sep=0pt] at (0,0) {$1,3$\rule[-1.5mm]{0pt}{5.5mm}};
    \node (r3) [circle, draw, minimum width=6mm,inner sep=1pt] at (1.4,0.8) {$7$};
    \node (A) [rectangle, draw, minimum width=1cm, inner sep=0pt] at (1.4,0) {$5,6$\rule[-1.5mm]{0pt}{5.5mm}};
    \draw[line width=0.5pt] (r5.south)--(r5|-r1.north) (r4.south)-- (r4|-r1.north) (r1)--(A)--(r3);
    \end{tikzpicture}
$$ 
\red{also lies in (\ref{notrivial}). The proof will then be complete once we} rule out any possible solution \red{$\alpha,\beta\in H^*(M^{(3)}_d(n))$} to the equation
    \begin{equation}\label{generalindeterminacy}
    \begin{tikzpicture}[baseline=7.2pt]
    \node (r4) [circle, draw, minimum width=6mm,inner sep=1pt] at (-0.35,0.8) {$2$};
    \node (r5) [circle, draw, minimum width=6mm,inner sep=1pt] at (0.35,0.8) {$4$};
    \node (r1) [rectangle, draw, minimum width=1cm, inner sep=0pt] at (0,0) {$1,3$\rule[-1.5mm]{0pt}{5.5mm}};
    \node (r3) [circle, draw, minimum width=6mm,inner sep=1pt] at (1.4,0.8) {$7$};
    \node (A) [rectangle, draw, minimum width=1cm, inner sep=0pt] at (1.4,0) {$5,6$\rule[-1.5mm]{0pt}{5.5mm}};
    \draw[line width=0.5pt] (r5.south)--(r5|-r1.north) (r4.south)-- (r4|-r1.north) (r1)--(A)--(r3);
    \end{tikzpicture}
=\alpha\cdot\,\begin{tikzpicture}[baseline=3pt]
    \node (A) [minimum size=1.1cm, draw, minimum width=1.1cm, inner sep=0pt] at (1.4,0.23) {${1,2,3}$\rule[-1.5mm]{0pt}{5.5mm}};
    \end{tikzpicture}\,\,+\beta\,\left(\ \rule{0mm}{8mm}
    \begin{tikzpicture}[baseline=3pt]
    \node (A) [minimum size=1cm, draw, minimum width=1cm, inner sep=0pt] at (1.4,0.23) {$\red{T_4}$\rule[-1.5mm]{0pt}{5.5mm}};
    \end{tikzpicture}\hspace{.5mm}+\hspace{.5mm}\begin{tikzpicture}[baseline=3pt]
    \node (A) [minimum size=1cm, draw, minimum width=1cm, inner sep=0pt] at (1.4,0.23) {$\red{T_5}$\rule[-1.5mm]{0pt}{5.5mm}};
    \end{tikzpicture}
\hspace{.5mm}+\hspace{.5mm}\begin{tikzpicture}[baseline=3pt]
    \node (A) [minimum size=1cm, draw, minimum width=1cm, inner sep=0pt] at (1.4,0.23) {$\red{T_6}$\rule[-1.5mm]{0pt}{5.5mm}};
    \end{tikzpicture}\hspace{.5mm}+\hspace{.5mm}\begin{tikzpicture}[baseline=3pt]
    \node (A) [minimum size=1cm, draw, minimum width=1cm, inner sep=0pt] at (1.4,0.23) {$\red{T_7}$\rule[-1.5mm]{0pt}{5.5mm}};
    \end{tikzpicture}\,\ 
    \right)\red{.}
    \end{equation}
\red{Here, for $\ell \in \{4,5,6,7\}$, $T_\ell$ stands for the ordered triple of elements in the set} $\red{\{T_\ell\}}:=\{4,5,6,7\}\setminus \{\ell\}$.

\red{In what follows, the expression of a cohomology class $\gamma\in H^*(\mdkn)$ as a $\mathbb{Z}_2$-linear combination of basic elements (Defin\blue{i}tion \ref{bosquesbasicos}) will be referred to as the \emph{expansion} of $\gamma$. For instance,} by dimensional reasons, \red{the expansions of both classes $\alpha$ and $\beta$ in a potential equation (\ref{generalindeterminacy}) would have to involve exclusively elementary basis elements} of the form
\begin{equation}\label{generalterm}
\begin{tikzpicture}[baseline=3mm]
    \node (r1) [circle, draw, minimum width=6mm,inner sep=1pt] at (-0.7,0.8) {$c$};
    \node (r2) [circle, draw, minimum width=6mm,inner sep=1pt] at (0,0.8) {$d$};
    \node (r3) [circle, draw, minimum width=6mm,inner sep=1pt] at (0.7,0.8) {$e$};
    \node (A) [rectangle, draw, minimum width=1.5cm, inner sep=0pt] at (0,0) {$a,b$\rule[-1.5mm]{0pt}{5.5mm}};
    \draw[line width=0.5pt] (r1.south)--(r1|-A.north) (r2.south)--(r2|-A.north) 
    (r3.south)--(r3|-A.north);
\end{tikzpicture}
\end{equation}
\red{with} $\{a,b,c,d,e\}\subset \nn$. \red{Then, looking at the expansions of the two products on the right hand-side of any such expression (\ref{generalindeterminacy}), and using} Lemma~\ref{inprepa}, \red{we see that no basis element $\varepsilon$ in the expansion of the first product of (\ref{generalindeterminacy}) can also appear in the expansion of the second product, nor $\varepsilon$ can be the basic element on the left hand-side of (\ref{generalindeterminacy}). Consequently,} the first product on the right-hand side of \red{any potential expression (\ref{generalindeterminacy}) would have to} vanish, and we \red{are left} to rule out solutions to the simpler equation
    \begin{equation}\label{generalindeterminacysimpler}
    \begin{tikzpicture}[baseline=7.2pt]
    \node (r4) [circle, draw, minimum width=6mm,inner sep=1pt] at (-0.35,0.8) {$2$};
    \node (r5) [circle, draw, minimum width=6mm,inner sep=1pt] at (0.35,0.8) {$4$};
    \node (r1) [rectangle, draw, minimum width=1cm, inner sep=0pt] at (0,0) {$1,3$\rule[-1.5mm]{0pt}{5.5mm}};
    \node (r3) [circle, draw, minimum width=6mm,inner sep=1pt] at (1.4,0.8) {$7$};
    \node (A) [rectangle, draw, minimum width=1cm, inner sep=0pt] at (1.4,0) {$5,6$\rule[-1.5mm]{0pt}{5.5mm}};
    \draw[line width=0.5pt] (r5.south)--(r5|-r1.north) (r4.south)-- (r4|-r1.north) (r1)--(A)--(r3);
    \end{tikzpicture}=\beta\,\left(\rule{0mm}{8mm}\ 
    \begin{tikzpicture}[baseline=3pt]
    \node (A) [minimum size=1cm, draw, minimum width=1cm, inner sep=0pt] at (1.4,0.23) {$\red{T}_4$\rule[-1.5mm]{0pt}{5.5mm}};
    \end{tikzpicture}\;\;+\;\;\begin{tikzpicture}[baseline=3pt]
    \node (A) [minimum size=1cm, draw, minimum width=1cm, inner sep=0pt] at (1.4,0.23) {$\red{T}_5$\rule[-1.5mm]{0pt}{5.5mm}};
    \end{tikzpicture}\;\;
+\;\;\begin{tikzpicture}[baseline=3pt]
    \node (A) [minimum size=1cm, draw, minimum width=1cm, inner sep=0pt] at (1.4,0.23) {$\red{T}_6$\rule[-1.5mm]{0pt}{5.5mm}};
    \end{tikzpicture}\;\;+\;\;\begin{tikzpicture}[baseline=3pt]
    \node (A) [minimum size=1cm, draw, minimum width=1cm, inner sep=0pt] at (1.4,0.23) {$\red{T}_7$\rule[-1.5mm]{0pt}{5.5mm}};
    \end{tikzpicture}\,\ 
    \right)\,\raisebox{-10pt}{.}
    \end{equation}

\red{Let $\mathfrak{S}$ denote the sum inside the parenthesis of (\ref{generalindeterminacysimpler}), and suppose for a contradiction that (\ref{generalindeterminacysimpler}) holds for some $\beta\in H^*(M^{(3)}_d(n))$. We can assume without loss of generality} that \red{no basis element} (\ref{generalterm}) in \red{the expansion of} $\beta$ \red{has} zero product with \red{$\mathfrak{S}$. In particular, a new application of} Lemma~\ref{inprepa} \red{shows that} all \red{basis elements} (\ref{generalterm}) \red{in the expansion } of $\beta$ \red{must satisfy}
\begin{equation}\label{fozados}
\{1,2,3\}\subset\{a,b,c,d,e\}\subset\{1,2,3,4,5,6,7\}.
\end{equation}

\red{Let us} analyze the product with \red{$\mathfrak{S}$ of} any such basis element \red{$\tau$ in the expansion of $\beta$.} 
To begin with, there must be a (not necessarily unique) $\ell\in\{4,5,6,7\}$ with
\begin{equation}\label{tauproduct}
\raisebox{-3mm}{\raisebox{4mm}{$\tau\,\,\cdot\,\,\, $}
\begin{tikzpicture}
 \node [minimum size=1cm, draw, minimum width=1cm, inner sep=0pt] at (1.4,0.23) {$\red{T}_\ell$\rule[-1.5mm]{0pt}{5.5mm}};
\end{tikzpicture}\;\raisebox{3.5mm}{$\neq0$.}}
\end{equation}

In particular, $\{a,b\}\cap \red{\{T_\ell\}}=\varnothing$ and $|\{c,d,e\}\cap \red{\{T_\ell\}}|=1$, in view of~(\ref{fozados}). Say $e\in \red{\{T_\ell\}}$, so \red{that} $\{c,d\}\cap \red{\{T_\ell\}}=\varnothing$. Then (\ref{tauproduct}) takes the (perhaps non-basic) form
%
\begin{equation*}
    \begin{tikzpicture}[baseline=7.2pt]
    \node (r4) [circle, draw, minimum width=6mm,inner sep=1pt] at (-0.35,0.8) {$c$};
    \node (r5) [circle, draw, minimum width=6mm,inner sep=1pt] at (0.35,0.8) {$d$};
    \node (r1) [rectangle, draw, minimum width=1cm, inner sep=0pt] at (0,0) {$a,b$\rule[-1.5mm]{0pt}{5.5mm}};
    \node (A) [minimum size=1cm, draw, minimum width=1cm, inner sep=0pt] at (1.4,0.23) {$\red{T_\ell}$\rule[-1.5mm]{0pt}{5.5mm}};
    \draw[line width=0.5pt] (r5.south)--(r5|-r1.north) (r4.south)-- (r4|-r1.north) (r1)--(r1-|A.west);
    \end{tikzpicture}\ \raisebox{-4mm}{,}
\end{equation*}

where $\{1,2,3\}\subset\{a,b,c,d\}\subset\{1,2,3,4,5,6,7\}\setminus \red{\{T_\ell\}}$. \red{Altogether, we have}
\begin{equation}\label{acotados}
\{a,b,c,d\}=\{1,2,3,\ell\} \quad \text{and} \quad \red{e\in\{T_\ell\},}
\end{equation}
\red{which} allows us to evaluate in full the product $\tau\cdot\red{\mathfrak{S}}$:
\begin{itemize}
    \item If $\ell\in \{a,b\}$, say $b=\ell$, \red{then}
    \begin{align*}
    \tau\cdot \red{\mathfrak{S}} &= \tau\cdot \,\ 
    \begin{tikzpicture}[baseline=3pt]
    \node (A) [minimum size=1cm, draw, minimum width=1cm, inner sep=0pt] at (1.4,0.23) {$\red{T}_\ell$\rule[-1.5mm]{0pt}{5.5mm}};
    \end{tikzpicture}\ =\begin{tikzpicture}[baseline=3pt]
    \node (r4) [circle, draw, minimum width=6mm,inner sep=1pt] at (-0.35,0.8) {$c$};
    \node (r5) [circle, draw, minimum width=6mm,inner sep=1pt] at (0.35,0.8) {$d$};
    \node (r1) [rectangle, draw, minimum width=1cm, inner sep=0pt] at (0,0) {$a,\ell$\rule[-1.5mm]{0pt}{5.5mm}};
    \node (A) [minimum size=1cm, draw, minimum width=1cm, inner sep=0pt] at (1.4,0.23) {$\red{T}_\ell$\rule[-1.5mm]{0pt}{5.5mm}};
    \draw[line width=0.5pt] (r5.south)--(r5|-r1.north) (r4.south)-- (r4|-r1.north) (r1)--(r1-|A.west);
    \end{tikzpicture}\\&=\rule{0mm}{12mm}\begin{tikzpicture}[baseline=3pt]
    \node (r4) [circle, draw, minimum width=6mm,inner sep=1pt] at (-0.35,0.8) {$d$};
    \node (r5) [circle, draw, minimum width=6mm,inner sep=1pt] at (0.35,0.8) {$\ell$};
    \node (r1) [rectangle, draw, minimum width=1cm, inner sep=0pt] at (0,0) {$a,c$\rule[-1.5mm]{0pt}{5.5mm}};
    \node (A) [minimum size=1cm, draw, minimum width=1cm, inner sep=0pt] at (1.4,0.23) {$\red{T}_\ell$\rule[-1.5mm]{0pt}{5.5mm}};
    \draw[line width=0.5pt] (r5.south)--(r5|-r1.north) (r4.south)-- (r4|-r1.north) (r1)--(r1-|A.west);
    \end{tikzpicture}\;+\,\begin{tikzpicture}[baseline=3pt]
    \node (r4) [circle, draw, minimum width=6mm,inner sep=1pt] at (-0.35,0.8) {$c$};
    \node (r5) [circle, draw, minimum width=6mm,inner sep=1pt] at (0.35,0.8) {$\ell$};
    \node (r1) [rectangle, draw, minimum width=1cm, inner sep=0pt] at (0,0) {$a,d$\rule[-1.5mm]{0pt}{5.5mm}};
    \node (A) [minimum size=1cm, draw, minimum width=1cm, inner sep=0pt] at (1.4,0.23) {$\red{T}_\ell$\rule[-1.5mm]{0pt}{5.5mm}};
    \draw[line width=0.5pt] (r5.south)--(r5|-r1.north) (r4.south)-- (r4|-r1.north) (r1)--(r1-|A.west);
    \end{tikzpicture}\ \raisebox{-2mm}{,}
    \end{align*}
    which is a sum of two basis elements, in view of~(\ref{acotados}).
    \item If $\ell\in\{c,d\}$, say $d=\ell$, \red{then for $j\in\{4,5,6,7\}\setminus\{d,e\}$}
$$\red{    
\tau\cdot \,\ 
    \begin{tikzpicture}[baseline=3pt]
    \node (A) [minimum size=1cm, draw, minimum width=1cm, inner sep=0pt] at (1.4,0.23) {$T_j$\rule[-1.5mm]{0pt}{5.5mm}};
    \end{tikzpicture}=0,
}$$
\red{so that}
    \begin{align*}\label{basicos}\tau \cdot \red{\mathfrak{S}} &= \tau\cdot \,\ 
    \begin{tikzpicture}[baseline=3pt]
    \node (A) [minimum size=1cm, draw, minimum width=1cm, inner sep=0pt] at (1.4,0.23) {$\red{T}_\ell$\rule[-1.5mm]{0pt}{5.5mm}};
    \end{tikzpicture}\ +\tau\cdot \,\ 
    \begin{tikzpicture}[baseline=3pt]
    \node (A) [minimum size=1cm, draw, minimum width=1cm, inner sep=0pt] at (1.4,0.23) {$\red{T}_e$\rule[-1.5mm]{0pt}{5.5mm}};
    \end{tikzpicture}\\&=\rule{0mm}{12mm}
    \begin{tikzpicture}[baseline=3pt]
    \node (r4) [circle, draw, minimum width=6mm,inner sep=1pt] at (-0.35,0.8) {$c$};
    \node (r5) [circle, draw, minimum width=6mm,inner sep=1pt] at (0.35,0.8) {$\ell$};
    \node (r1) [rectangle, draw, minimum width=1cm, inner sep=0pt] at (0,0) {$a,b$\rule[-1.5mm]{0pt}{5.5mm}};
    \node (A) [minimum size=1cm, draw, minimum width=1cm, inner sep=0pt] at (1.4,0.23) {$\red{T}_\ell$\rule[-1.5mm]{0pt}{5.5mm}};
    \draw[line width=0.5pt] (r5.south)--(r5|-r1.north) (r4.south)-- (r4|-r1.north) (r1)--(r1-|A.west);
    \end{tikzpicture}\enskip 
    +\enskip 
    \begin{tikzpicture}[baseline=3pt]
    \node (r4) [circle, draw, minimum width=6mm,inner sep=1pt] at (-0.35,0.8) {$c$};
    \node (r5) [circle, draw, minimum width=6mm,inner sep=1pt] at (0.35,0.8) {$e$};
    \node (r1) [rectangle, draw, minimum width=1cm, inner sep=0pt] at (0,0) {$a,b$\rule[-1.5mm]{0pt}{5.5mm}};
    \node (A) [minimum size=1cm, draw, minimum width=1cm, inner sep=0pt] at (1.4,0.23) {$\red{T}_e$\rule[-1.5mm]{0pt}{5.5mm}};
    \draw[line width=0.5pt] (r5.south)--(r5|-r1.north) (r4.south)-- (r4|-r1.north) (r1)--(r1-|A.west);
    \end{tikzpicture}\ ,
    \end{align*}
    \red{which is again} a sum of two basis elements, in view of~(\ref{acotados}).
\end{itemize}

This shows that the product $\beta\cdot \red{\mathfrak{S}}$ is a sum of an even number of \red{basis} elements, \red{which is incompatible with} (\ref{generalindeterminacysimpler})\blue{, since the left-hand side term is its own expansion}. %
\end{proof}


\begin{corollary}\label{muygeneral} For seven pairwise distinct numbers  $a$, $b$, $c$, $d$, $e$, $f$, $g$ in $\nn$, 
the \red{triple} Massey product in $M_d^{(3)}(n)$
\[
\left\langle
\elementary{a}{b}{c},\elementary{c}{d}{e},\elementary{e}{f}{g}{+}\elementary{d}{f}{g}{+}\elementary{d}{e}{g}{+}\elementary{d}{e}{f}
\right\rangle
\]
is \red{well-defined, non-trivial, and represented by}
$$
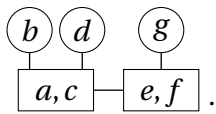
\begin{tikzpicture}[baseline=3mm]
    \node (r4) [circle, draw, minimum width=6mm,inner sep=1pt] at (-0.35,0.8) {$b$};
    \node (r5) [circle, draw, minimum width=6mm,inner sep=1pt] at (0.35,0.8) {$d$};
    \node (r1) [rectangle, draw, minimum width=1cm, inner sep=0pt] at (0,0) {$a,c$\rule[-1.5mm]{0pt}{5.5mm}};
    \node (r3) [circle, draw, minimum width=6mm,inner sep=1pt] at (1.4,0.8) {$g$};
    \node (A) [rectangle, draw, minimum width=1cm, inner sep=0pt] at (1.4,0) {$e,f$\rule[-1.5mm]{0pt}{5.5mm}};
    \draw[line width=0.5pt] (r5.south)--(r5|-r1.north) (r4.south)-- (r4|-r1.north) (r1)--(A)--(r3);
    \end{tikzpicture}\;\,\raisebox{-5mm}{\red{.}}
    $$
\end{corollary}
\begin{proof}
\red{Choose a permutation $\sigma\in\Sigma_n$ with $\sigma(1)=a$, $\sigma(2)=b$, $\sigma(3)=c$, $\sigma(4)=d$, $\sigma(5)=e$, $\sigma(6)=f$ and $\sigma(7)=g$. Then, the induced diffeomorphism $\widetilde{\sigma}\colon M^{(3)}_d(n)\to  M^{(3)}_d(n)$ identifies the triple Massey product in Theorem \ref{ptmnt} with the one in Corollary \ref{muygeneral}.}
\end{proof}


\bigskip
{\small \sc Departamento de Matem\'aticas\\
Centro de Investigaci\'on y de Estudios Avanzados del I.P.N.\\
Av.~Instituto Polit\'ecnico Nacional n\'umero 2508\\
San Pedro Zacatenco, M\'exico City 07000, M\'exico \\
{\tt jesus@math.cinvestav.mx}\\\\
\blue{\small \sc Centro de Investigaci\'on en Matem\'aticas, A.C. Unidad M\'erida\\
Parque Cient\'ifico y Tecnol\'ogico de Yucat\'an\\
Carretera Sierra Papacal--Chuburn\'a Puerto Km 5.5 \\
Sierra Papacal, M\'erida, Yucat\'an, 97302, México \\
{\tt  luis.leon@cimat.mx}}}

\end{document}